\documentclass{amsart}
\usepackage{amssymb, enumitem, graphicx}
\usepackage[hidelinks]{hyperref}
\usepackage[all]{xy}

\newtheorem{thm}{Theorem}[section]
\newtheorem{cor}[thm]{Corollary}
\newtheorem{prop}[thm]{Proposition}
\newtheorem{lem}[thm]{Lemma}
\newtheorem{fac}[thm]{Fact}
\theoremstyle{definition}
\newtheorem{defn}[thm]{Definition}
\newtheorem{exa}[thm]{Example}
\newtheorem{exas}[thm]{Examples}
\newtheorem{rem}[thm]{Remark}
\newtheorem{prob}{Problem}

\newcommand{\eg}{\emph{e.g.}}
\newcommand{\ie}{\emph{i.e.}}

\newcommand{\smsum}{\textstyle\sum\limits}
\let\on\operatorname
\let\phi\varphi

\begin{document}

\title[Congruence-semisimple semirings and $K_{0}$-group characterization]%
{On congruence-semisimple semirings and the $K_{0}$-group characterization
  of ultramatricial algebras over semifields}

\author{Yefim~Katsov}
\address{Department of Mathematics \\
  Hanover College, Hanover, IN 47243--0890, USA}
\email{katsov@hanover.edu}

\author{Tran Giang Nam}
\address{Institute of Mathematics, VAST \\
  18 Hoang Quoc Viet, Cau Giay, Hanoi, Vietnam}
\email{tgnam@math.ac.vn} 

\author{Jens Zumbr\"agel}
\address{Faculty of Computer Science and Mathematics \\
  University of Passau, Germany}
\email{jens.zumbraegel@uni-passau.de}

\thanks{The second author is supported by the Vietnam National
  Foundation for Science and Technology Development (NAFOSTED)
  under Grant 101.04-2017.19.}


\begin{abstract}\sloppy
  In this paper, we provide a complete description of
  congruence-semisimple semirings and introduce the pre-ordered
  abelian Grothendieck groups $K_{0}(S)$ and $SK_{0}(S)$ of the
  isomorphism classes of the finitely generated projective and
  strongly projective $S$-semimodules, respectively, over an arbitrary
  semiring~$S$.  We prove that the $SK_{0}$-groups and $K_{0}$-groups
  are complete invariants of, \ie, completely classify, ultramatricial
  algebras over a semifield~$F$.  Consequently, we show that the
  $SK_{0}$-groups completely characterize zerosumfree
  congruence-semisimple semirings. \medskip
\end{abstract}

\maketitle

\section{Introduction}

As is well-known (see, for example, \cite{bass:akt}), projective
modules play a fundamental role in developing of algebraic $K$-theory
which, in turn, has crucial outcomes in many areas of modern
mathematics such as topology, geometery, number theory, functional
analysis, etc.  In short, algebraic $K$-theory is a study of groups of
the isomorphism classes of algebraic objects, the first of which is
$K_{0}(R)$, Grothendieck's group of the isomorphism classes of
finitely generated projective $R$-modules, and that is used to create
a sort of dimension for $R$-modules that lack a basis.  Therefore, the
structure theory of projective modules is certainly of a great
interest and importance.

Semirings, semimodules, and their applications, arise in various
branches of mathematics, computer science, quantum physics, and many
other areas of science (see, for example, \cite{golan:sata} and
\cite{glazek:agttlos}).  As algebraic structures, semirings certainly
are the most natural generalization of such (at first glance
different) algebraic concepts as rings and bounded distributive
lattices, and therefore, they form a very natural and exciting ground
for furthering the structure theory of projective (semi)modules in a
``non-additive'' categorical setting.  And, in fact, the structure
theory of projective semimodules has been recently considered by
several authors (see, \eg, \cite{kat:thcos}, \cite{patch:psoswvini},
\cite{kn:meahcos}, \cite{ik:ospopsops}, \cite{ijk:pdapotp}, \cite%
{mac:pmops}, \cite{ikr:domlzs} and \cite{ikn:thstososaowcsap}).  Also,
in the last one or two decades, there can be observed an
intensively growing substantial interest in additively idempotent
semirings, which particularly include the Boolean and tropical
semifields and have a fundamental meaning in such relatively new,
``non-traditional'', and fascinating areas of modern mathematics such
as tropical geometry~\cite{rst:fsitg} and~\cite{gg:eotv}, tropical/%
supertropical algebra~\cite{ir:sa}, $\mathbb{F}_{1}$-geometry~\cite%
{cc:sofazf}, and the geometry of blueprints~\cite{lor:tgob}.

Although, in general, describing the structure of (finitely generated)
projective semimodules seems to be a quite difficult task, recently
there have been obtained a number of interesting results regarding
structures of projective semimodules over special classes of semirings
among which we mention, for example, the following ones.  Il'in et
al.~\cite{ikn:thstososaowcsap} initiated a homological structure
theory of semirings and investigated semirings all of whose cyclic
semimodules are projective; Izhakian et al.\ \cite{ijk:pdapotp}
characterized finitely generated projective semimodules over a
tropical semifield in terms of rank functions of semimodules; and
Macpherson~\cite{mac:pmops} classified projective semimodules over
additively idempotent semirings that are free on a monoid.
Further, motivated by direct sum decompositions of subsemimodules of
free semimodules over a tropical semifield and related structures,
Izhakian et al.\ \cite{ikr:domlzs} developed a theory of the
decomposition socle, $\text{dsoc}(M)$, for zerosumfree semimodules~$M$.
In particular, they provided a criterion for zerosumfree semirings~$S$
when $\text{dsoc}(S)=S$ (\cite[Thm.~3.3]{ikr:domlzs}) and established
the uniqueness of direct sum decompositions for some special finitely
generated projective semimodules (\cite[Cor.~3.4]{ikr:domlzs}), called
`\emph{strongly projective}' in the present paper, over such semirings.

Moreover, Elliott \cite{elliott:otcoilososfa} classified/characterized
ulramatricial algebras over an arbitrary field by means of their
pointed ordered Grothendieck groups~$K_{0}$.  This fundamental result
implies a $C^{\ast}$-algebra technique and initiated very fruitful
research lines in algebra and operator algebra, not to mention that
the Elliott program of classifying simple nuclear separable
$C^{\ast}$-algebras by $K$-theoretic invariants became a profoundly
active area of research (see, \eg, the survey paper by Elliott and
Toms~\cite{et:rpitcpfsaa}).

In light of the two previous paragraphs and motivated by the Elliott
program of classifying $C^{\ast}$-algebras in terms of $K$-theory, our
paper has a twofold goal: to characterize the decomposition socles and
structure of (finitely generated) projective semimodules over a
semiring~$S$ in terms of the Grothendieck group $K_{0}(S)$ of a
semiring~$S$; and to extend Elliott's classification of ultramatricial
algebras to a ``non-additive'' semiring setting.  Let us a briefly
clarify the latter: If~$F$ is a (semi)field and~$\mathcal{C}$ is a class
of unital $F$-algebras, then one says that the $K_{0}$-group is a
\emph{complete invariant} for algebras in~$\mathcal{C}$, or that $K_{0}$
\emph{completely classifies} $F$-algebras in~$\mathcal{C}$, if
any $F$-algebras~$R$ and~$S$ from $\mathcal{C}$ are isomorphic as
$F$-algebras iff there is a group isomorphism $K_{0}(R) \cong K_{0}(S)$
which respects the natural pre-order structure of the $K_{0}$-groups
and their order-units.  It should be mentioned that the ``blueprints''
of Lorscheid \cite{lor:tgob} contain commutative semirings as a full
subcategory, which eventually leads to a $K$-theory of blueprints,
including a $K$-theory of commutative semirings as a special case, and
the group $K_{0}(S)$ of a semiring~$S$ has been introduced by Di Nola
and Russo \cite{dr:tstatmaas}.  However, the considerations of~$K_{0}(S)$
in our paper are distinguished from those in \cite{lor:tgob} and
\cite{dr:tstatmaas} --- we consider two quite different types of
$K_{0}$-groups and, to the extend of our knowledge, at the first
time use them as complete invariants for classifying algebras of a
non-additive category in the spirit of the Elliot program.

The article is organized as follows.  In Section~\ref{sec:2}, for the
reader's convenience, we briefly collect the necessary notions and
facts on semirings and semimodules.  Subsequently, we provide in
Section~\ref{sec:3} a full description of congruence-semisimple
semirings (Theorem~\ref{thm:34}) and show that zerosumfree
congruence-semi\-simple semirings are precisely matricial algebras
over the Boolean semifield~$\mathbb{B}$ (Corollary~\ref{cor:35}).  In
Section~\ref{sec:4}, beyond of some basic considerations of strongly
projective semimodules under change of semirings
(Propositions~\ref{prop:48} and~\ref{prop:410}), we give a complete
description of the strongly projective semimodules over an arbitrary
semisimple semiring (Theorems~\ref{thm:45} and~\ref{thm:49}).

Based on the results of Section~\ref{sec:4} and \cite[Ch.~15]{g:vnrr},
in Section~\ref{sec:5} we introduce and establish fundamental
properties of the monoids $\mathcal{V}(S)$ ($\mathcal{SV}(S)$) of the
isomorphism classes of finitely generated (strongly) projective
semimodules over a semiring~$S$ and show that those monoids completely
characterize the class of ultramatricial algebras over a semifield $F$
(Theorems~\ref{thm:510} and~\ref{thm:511}).

In Section~\ref{sec:6}, using the results of Sections~\ref{sec:4}
and~\ref{sec:5}, we consider the pre-ordered abelian groups $K_{0}(S)$
and $SK_{0}(S)$---which are the Grothendieck groups on the monoids
$\mathcal{V}(S)$ and $\mathcal{SV}(S)$, respectively---for an
arbitrary semiring~$S$, and, using the concept of `\emph{weak
  dimension}' of semimodules, describe division semirings~$D$ having
the groups $K_{0}(D)$ and $SK_{0}(D)$ to be isomorphic
(Theorem~\ref{thm:610}).  Also, it is shown
(Proposition~\ref{prop:67}) that, for any additively idempotent
commutative semiring~$S$, the group $K_{0}(S)$ always contains a free
abelian group with countably infinite basis; and it is given
(Theorem~\ref{thm:614}) for semirings~$S$ having $\text{dsoc}(S)=S$,
\ie, \emph{congruence-semisimple} semirings here, a $K$-theory version
of \cite[Cor.~3.4]{ikr:domlzs}.  Finally, we extend Elliott's
classification theorem for ultramatricial algebras over fields
\cite{elliott:otcoilososfa} and show that the $SK_{0}$-groups and
$K_{0}$-groups are complete invariants of ultramatricial algebras over
semifields (Theorems~\ref{thm:621} and~\ref{thm:623}), as well as that
$SK_{0}$ completely classifies zerosumfree congruence-semisimple
semrings (Theorem~\ref{thm:622}).

All notions and facts of categorical algebra, used here without any
comments, can be found in \cite{macl:cwm}; for notions and facts from
semiring theory we refer to \cite{golan:sata}.

\section{Preliminaries}\label{sec:2}

Recall~\cite{golan:sata} that a \emph{semiring} is an algebra
$(S, +, \cdot, 0, 1)$ such that the following conditions are satisfied:
\begin{enumerate}[label=(\arabic*)]
\item $(S, +, 0)$ is a commutative monoid with identity element~$0$;
\item $(S, \cdot, 1)$ is a monoid with identity element~$1$;
\item multiplication distributes over addition from either side;
\item $0 s = 0 = s 0$ for all $s \in S$.
\end{enumerate}
Given two semirings~$S$ and $S'$, a map $\phi \colon S \to S'$ is a
\emph{homomorphism} if it satisfies $\phi(0) = 0$, $\phi(1) = 1$,
$\phi(x + y) = \phi(x) + \phi(y)$ and $\phi(x y) = \phi(x)\phi(y)$ for
all $x, y\in S$.

A semiring~$S$ is \emph{commutative} if $(S, \cdot, 0)$ is a
commutative monoid; and~$S$ is \emph{entire} if $a b = 0$ implies that
$a = 0$ or $b = 0$ for all $a, b \in S$.  The semiring~$S$ is a
\emph{division semiring} if $(S \!\setminus\! \{ 0 \}, \cdot, 1)$ is a
group; and~$S$ is a \emph{semifield} if it is a commutative division
semiring.  An element~$e$ in a given semiring~$S$ is \emph{idempotent}
if $e^2 = e$; and an idempotent $e \in S$ is \emph{strong} if there
exists an idempotent $f \in S$ such that $e + f = 1$ and
$e f = 0 = f e$.  Two idempotents $e, f \in S$ are \emph{orthogonal}
if $e f = 0 = f e$.  An idempotent is \emph{primitive} if it cannot be
written as the sum of two nonzero orthogonal idempotents.

As usual, a right $S$-\emph{semimodule} over a given semiring~$S$ is a
commutative monoid $(M, +, 0_{M})$ together with a scalar
multiplication $(m, s) \mapsto m s$ from $M \times S$ to~$M$ which
satisfies the identities $m (s s') = (m s) s'$,
$(m + m') s = m s + m' s$, $m (s + s') = m s + m s'$, $m 1 = m$,
$0_{M} s = 0_{M} = m 0$ for all $s, s' \in S$ and $m, m' \in M$.  Left
semimodules over~$S$ and homomorphisms between semimodules are defined
in the standard manner.  An $S$-semimodule~$M$ is called a
\emph{module} if its additive reduct $(M, +, 0_{M})$ is an abelian
group.  Let henceforth $\mathcal{M}$ be the variety of commutative
monoids, and let $\mathcal{M}_{S}$ and $_{S} \mathcal{M}$ denote the
categories of right and left $S$-semimodules, respectively, over a
semiring~$S$.

Recall~\cite[Def.~3.1]{kat:tpaieosoars} the tensor product bifunctor
$- \otimes - \colon \mathcal{M}_{S} \times _{S} \mathcal{M} \to
\mathcal{M}$, which for a right semimodule $A \in |\mathcal{M}_{S}|$
and a left semimodule $B \in |_{S}\mathcal{M}|$ can be described as
the factor monoid $F / \sigma$ of the free monoid $F \in |\mathcal{M}|$,
generated by the Cartesian product $A \times B$, factorized with
respect to the congruence~$\sigma$ on~$F$ generated by the ordered
pairs having the form
\[ \langle (a_{1} + a_{2}, b), (a_{1}, b) + (a_{2}, b) \rangle , \
  \langle (a, b_{1} + b_{2}), (a, b_{1}) + (a, b_{2}) \rangle , \
  \langle (a s, b), (a, s b) \rangle , \]
with $a_{1}, a_{2} \in A$, $b_{1}, b_{2} \in B$ and $s \in S$.

An $S$-semimodule~$M$ is called \emph{(additively) idempotent} (resp.,
\emph{zerosumfree}) if $m + m = m$ for all $m \in M$ (resp., if
$m + m' = 0$ implies $m = m' = 0$ for all $m, m' \in M$); clearly,
every idempotent semimodule is zerosumfree.  In particular, a
semiring~$S$ is \emph{additively idempotent} (resp.,
\emph{zerosumfree}, a \emph{ring}) if $S_{S} \in |\mathcal{M}_{S}|$
as a semimodule is idempotent (resp., zerosumfree, a module).  Two
well-known important examples of additively idempotent semirings are
the \emph{Boolean semifield} $\mathbb{B} := (\{ 0, 1 \}, \max,
\min, 0, 1)$ and the \emph{tropical semifield} $\mathbb{T} :=
(\mathbb{R} \cup \{-\infty\}, \max, +, -\infty, 0)$.

By an $S$-algebra~$A$ over a given commutative semiring~$S$ we mean
the data of an $S$-semimodule~$A$ and of an associate multiplication
on~$A$ that is bilinear with respect to the operations of the
$S$-semimodule~$A$.  For example, every semiring may be considered as
a $\mathbb{Z}^+$-algebra and any additively idempotent semiring as a
$\mathbb{B}$-algebra.  An $S$-algebra~$A$ is called \emph{unital} if
the multiplication on~$A$ has a neutral element $1_A$, \ie,
$a 1_A = a = 1_A a$ for all $a \in A$.

As usual (see, for example, \cite[Ch.~17]{golan:sata}), if~$S$ is a
semiring, then in the category~$\mathcal{M}_{S}$, a \emph{free} right
semimodule~$F$ with basis set~$I$ is a direct sum (a coproduct) of
$I$~copies of $S_{S}$, \ie, $F = \bigoplus_{i\in I} S_{i}$ where
$S_{i} \cong S_{S}$ for $i\in I$.  Accordingly, a projective
semimodule in $\mathcal{M}_{S}$ is defined to be a retract of a free
semimodule, \ie, a right semimodule~$P$ is called \emph{projective} if
there is a free right semimodule~$F$ with homomorphisms $f \colon F
\to P$ and $g \colon P \to F$ such that $f \circ g = \on{id}_P$.  And
a semimodule~$M_{S}$ is \emph{finitely generated} if it is a
homomorphic image of a free semimodule with finite basis set.
Moreover, a semiring~$S$ is said to have the \emph{IBN} (``invariant
basis number'') property (cf.~\cite[Def.~2.8]{klp:aocofmfsaflm}) if, for
any natural numbers $n, m$, the free semimodules~$S^{m}$ and~$S^{n}$
are isomorphic in $\mathcal{M}_{S}$ if and only if $m = n$.  Note that
the ``left'' version of the IBN property is equivalent to this right
version, see~\cite[Prop.~3.1]{klp:aocofmfsaflm}.

\emph{Congruences} on a right $S$-semimodule~$M$ are defined in the
standard manner.  Any subsemimodule~$L$ of a right $S$-semimodule~$M$
induces a congruence~$\equiv_{L}$ on~$M$, known as the \emph{Bourne
  congruence}, by setting $m \equiv_{L} m'$ iff $m + l = m' + l'$ for
some $l, l' \in L$.  The subsemimodule~$L$ is \emph{subtractive} if
$a \in L$ and $x + a \in L$ (where $x \in M$) implies $x \in L$.  In
this case, the zero class~$[0]$ with respect to $\equiv_{L}$ coincides
with~$L$.  A nonzero right $S$-semimodule~$M$ is called
\emph{congruence-simple} if its only congruences are the
\emph{identity congruence} $\vartriangle_M \,:=
\{ (m, m) \mid m \in M \}$ and the \emph{universal congruence}
$M \!\times\! M$ --- in this case, its only subtractive subsemimodules
are $\{ 0 \}$ and~$M$.  Accordingly, a right ideal~$I$ of a
semiring~$S$ is called \emph{congruence-simple} if the right
$S$-semimodule~$I$ is congruence-simple.  And a semiring~$S$ is called
\emph{right (left) congruence-semisimple} if~$S$ is a direct sum of
congruence-simple right (left) ideals.

Finally, a nonzero right $S$-semimodule~$M$ is called \emph{minimal}
if it has no proper nonzero subsemimodules, and the $S$-semimodule~$M$
is said to be \emph{semisimple} if it is a direct sum of minimal
subsemimodules; in particular, a semiring~$S$ is said to be
\emph{right (left) semisimple} if the right (left) regular semimodule
is semisimple.  As is well-known (see, for example,
\cite[Thm.~7.8]{hebwei:hoa} or \cite[Thm.~4.5]{knt:ossss}), the
celebrated Artin-Wedderburn theorem generalized to semirings states
that a semiring~$S$ is (right, left) semisimple if and only if
\[ S \,\cong\, M_{n_{1}}(D_{1}) \times \ldots \times M_{n_{r}}(D_{r}) \,, \]
where $M_{n_{i}}(D_{i})$ is the semiring of $n_{i} \times n_{i}$-matrices
over a division semiring~$D_{i}$ for each $i = 1, \ldots, r$.  In the
sequel, we refer to such an isomorphism as a \emph{direct product
  representation} of a semisimple semiring~$S$.

\section{Congruence-semisimple semirings}\label{sec:3}

Providing a full description of the class of all congruence-semisimple
semirings constitutes a main goal of this section; and to accomplish
it, we need the following useful facts.

\begin{lem}[{\cite[Prop.~1.2]{i:vs}}]\label{lem:31}
  If $M \in |\mathcal{M}_{S}|$ is a congruence-simple right
  $S$-semimodule, then~$M$ is either idempotent or a module.
\end{lem}

\begin{lem}[{\cite[Lem.~1.1]{i:vs}}]\label{lem:32}
  Let $M \in |\mathcal{M}_{S}|$ be an idempotent right $S$-semimodule.
  Then the relation \[ x \sim y \quad\Longleftrightarrow\quad
  \on{Ann}(x) = \on{Ann}(y) , \]
  where $\on{Ann}(m) := \{s \in S \mid m s = 0_{M} \}$ is the
  annihilator of an element $m \in M$, is a congruence on~$M$.
\end{lem}

\begin{prop}\label{prop:33}
  Let~$S$ be a semiring.
  \begin{enumerate}[label=\upshape(\arabic*)]
  \item The endomorphism semiring $\on{End}_{S}(M)$ of any cyclic
    congruence-simple $S$-semimodule $M \in |\mathcal{M}_{S}|$ is
    either a division ring or the Boolean semifield~$\mathbb{B}$.
  \item Let~$I$ be a congruence-simple right ideal of a semiring~$S$
    such that $I = e S$ for some idempotent $e \in S$, and let
    $M \in |\mathcal{M}_{S}|$ be a cyclic congruence-simple
    $S$-semimodule.  Then, $M \cong I$ or $\on{Hom}_{S}(M, I) = 0$.
  \end{enumerate}
\end{prop}

\begin{proof}
  (1) By Lemma~\ref{lem:31}, $M$ is either a module or idempotent.  
  If~$M$ is a module, then using similar arguments as in the classical
  Schur lemma \cite[Lem.~1.3.6]{lam:afcinr}, one readily sees that
  $\on{End}_{S}(M)$ is a division ring.

  Now suppose that~$M$ is an idempotent semimodule such that $M = m S$
  for some $0 \ne m \in M$, and let $f \in \on{End}_{S}(M)$ be a
  nonzero endomorphism.  It is clear (see also \cite[Prop.~2.1
  (1)]{aikn:ovsasaowcsai}) that~$f$ is injective.  If $m^{\prime} :=
  f(m)$, then $\on{Ann}(m) = \on{Ann}(m^{\prime})$: Indeed, since~$f$
  is injective, for all $s \in S$, $m s=0$ iff $m^{\prime} s = f(m s)
  = 0$.  And, by Lemma~\ref{lem:32}, $m = m^{\prime} = f(m)$, thus
  $f = \on{id}_{M}$, whence $\on{End}_{S}(M) \cong \mathbb{B}$.

  (2) Assume there exists a nonzero homomorphism $f \colon M \to I$.
  Again, it is clear (see also \cite[Prop.~2.1 (1)]{aikn:ovsasaowcsai})
  that~$f$ is injective.  We claim that~$f$ is also surjective,
  whence $M \cong I$.  If~$M$ is a module, then $0 \ne f(M)$ is a
  subtractive submodule of~$I$, hence $f(M) = I$ as desired.
  Suppose then that~$M$ is not a module, hence~$M$ is idempotent
  by Lemma~\ref{lem:31}.  Then $0 \ne f(M)$ is also idempotent, thus~$I$
  is not a module either, so that~$I$ is idempotent by Lemma~\ref{lem:31}
  as well.  Now, let $M = m S$ for some $0 \ne m \in M$, and $a :=
  f(m) \in I$.  Since $f(M)$ is a nonzero subsemimodule of~$I$, the
  Bourne congruence $\equiv_{f(M)}$ on~$I$ is the universal one.  So,
  $e \equiv_{f(M)} 0$, and hence, $e + a s = a s^{\prime}$ and $e + a
  s e = a s^{\prime} e$ for some $s, s^{\prime}\in S$.  From the
  latter and since~$I$ is idempotent, one immediately gets that
  $\on{Ann}(e) = \on{Ann}(a s^{\prime} e)$; and by Lemma~\ref{lem:32},
  $e = a s^{\prime} e = f(m s^{\prime} e)$, and hence, $f$~is a again
  surjective.
\end{proof}

Recall that a semiring~$S$ is \emph{right (left) congruence-semisimple}
if~$S$ is a direct sum of its congruence-simple right (left) ideals.
Notice that a ring is a right (left) congruence-semisimple ring iff it
is a classical semisimple ring, \ie, it is a direct sum of its minimal
one-sided ideals.  However, in a semiring setting, this fact is not
true in general.  The following theorem, constituting the main result
of this section, gives a full description of all congruence-semisimple
semirings and also demonstrates that the class of congruence-semisimple
semirings is a proper subclass of the class of semisimple semirings.

\begin{thm}\label{thm:34}
  For any semiring~$S$, the following statements are equivalent:
  \begin{enumerate}[label=\upshape(\arabic*)]
  \item $S$ is a right congruence-semisimple semiring;
  \item $S \cong M_{n_{1}}(\mathbb{B}) \times \cdots \times
    M_{n_{k}}(\mathbb{B}) \times M_{m_{1}}(D_{1}) \times \cdots \times
    M_{m_{r}}(D_{r})$, where~$\mathbb{B}$ is the Boolean semifield,
    $D_{1}, \dots, D_{r}$ are division rings, $k \ge 0$, $r \ge 0$, and
    $n_{i}$~$(i = 1, \dots, k)$ and $m_{j}$~$(j = 1, \dots, r)$ are
    positive integers;
  \item $S$ is a left congruence-semisimple semiring.
  \end{enumerate}
\end{thm}

\begin{proof}
  $(1) \!\Longrightarrow\! (2)$.  Let~$S$ be a right congruence-semisimple
  semiring, thus~$S$ is a finite direct sum of its congruence-simple 
  right ideals.  By applying Lemma~\ref{lem:31} and grouping those summands
  according to their isomorphism types as right $S$-semimodules, we
  obtain \[ S_{S} \,\cong\, I_{1}^{n_{1}} \oplus \cdots \oplus I_{k}^{n_{k}}
  \oplus I_{k+1}^{m_{1}} \oplus \cdots \oplus I_{k+r}^{m_{r}} \,, \]
  with mutually nonisomorphic congruence-simple right ideals $I_{1},
  \dots, I_{k+r}$ of~$S$, where the~$I_{i}$ for $i = 1, \dots, k$ are
  additively idempotent $S$-semimodules and the additive reducts of
  $I_{i}$ for $i = k+1, \dots, k+r$ are abelian groups.

  Notice that each $I_{i}$ for $i = 1, \dots, k+r$ is a direct summand
  of $S_{S}$, so $I_{i} = e_{i} S$ for some idempotent $e_{i} \in S$.
  By Proposition~\ref{prop:33}, $\on{End}_{S}(I_{i}) \cong \mathbb{B}$ for 
  $i = 1, \dots, k$, $D_{j} := \on{End}_{S}(I_{k+j})$ for $j = 1, \dots, r$
  are division rings, and $\on{Hom}_{S}(I_{i}, I_{j}) = 0$ for all
  distinct $i, j = 1, \dots, k + r$.

  Since elements of $\on{End}_{S}(S_{S})$ are presented by
  multiplications on the left by elements of~$S$, and 
  as $\on{Hom}_{S}(I_{i}, I_{j}) = 0$ for $i \ne j$, we infer
  \begin{align*}
    S &\cong \on{End}_{S}(S_{S}) \cong 
    \on{End}_{S}(I_{1}^{n_{1}} \oplus \cdots \oplus I_{k}^{n_{k}}
    \oplus I_{k+1}^{m_{1}} \oplus \cdots \oplus I_{k+r}^{m_{r}}) \\ 
    &\cong \on{End}_{S}(I_{1}^{n_{1}}) \times \cdots \times
    \on{End}_{S}(I_{k}^{n_{k}}) \times \on{End}_{S}(I_{k+1}^{m_{1}})
    \times \cdots \times \on{End}_{S}(I_{m+r}^{m_{r}}) \,.
  \end{align*}
  Finally, by noting that $\on{End}_{S}(M^{m}) \cong M_{m}(\on{End}_{S}(M))$
  for any $M \in |\mathcal{M}_{S}|$ and postive integer~$m$, we conclude
  that \[ S \cong M_{n_{1}}(\mathbb{B}) \times \cdots \times 
  M_{n_{k}}(\mathbb{B}) \times M_{m_{1}}(D_{1}) \times \cdots 
  \times M_{m_{r}}(D_{r}) \,. \]

  $(2) \!\Longrightarrow\! (1)$.  It suffices to show the
  congruence-semisimpleness of a matrix semiring $S := M_{n}(K)$ 
  with~$K$ to be either a division ring or the Boolean semifield.  
  To this end, let $e_{ii}$ for $i = 1, \dots, n$ be the matrix units in
  $M_{n}(K)$, so that $S = e_{11} S \oplus \cdots \oplus e_{nn} S$ with
  $e_{ii} S \cong K^{n}$ as right $S$-semimodules for each~$i$.
  As was shown in \cite[Thm.~5.14]{kat:thcos}, the functors $F :
  \mathcal{M}_{S} \leftrightarrows \mathcal{M}_{K} : G$ given by $F(A)
  = A e_{11}$ and $G(B) = B^{n}$ establish an equivalence of the
  semimodule categories~$\mathcal{M}_{S}$ and~$\mathcal{M}_{K}$.
  Therefore, taking into consideration that~$K$ is a congruence-simple
  right $K$-semimodule and \cite[Lem.~3.8]{aikn:ovsasaowcsai}, we have
  that each $e_{ii} S \cong K^{n} = G(K)$ for $i = 1, \dots, n$ is a
  congruence-simple right $S$-semimodule as well, whence~$S$ is a
  right congruence-semisimple semiring.

  The equivalence $(1) \!\Longleftrightarrow\! (3)$ follows by symmetry.
\end{proof}

Recently in \cite{ikr:domlzs}, introducing and studying the
\emph{decomposition socle} for semimodules over zerosumfree semirings,
the authors characterize zerosumfree semirings~$R$ such that the
regular semimodule~$R_{R}$ is a finite direct sum of indecomposable
projective $R$-subsemimodules (see \cite[Thm.~3.3]{ikr:domlzs}).  In
light of this and as a corollary of Theorem~\ref{thm:34}, we note a
description of such semirings in the class of congruence-semisimple
semirings as matricial algebras---a central subject of our
considerations in the following sections---over the Boolean
semifield.

\begin{cor}\label{cor:35}
  For any zerosumfree semiring~$S$, the following are equivalent:
  \begin{enumerate}[label=\upshape(\arabic*)]
  \item $S$ is a right congruence-semisimple semiring;
  \item $S \cong M_{n_{1}}(\mathbb{B}) \times \cdots \times
    M_{n_{k}}(\mathbb{B})$, with~$\mathbb{B}$ the Boolean semifield
    and $k, n_{1}, \dots, n_{k}$ some positive integers;
  \item $S$ is a left congruence-semisimple semiring.
  \end{enumerate}
\end{cor}

\section{Strongly projective semimodules}\label{sec:4}

In~\cite{ikr:domlzs}, a very natural variation of the concept of a
projective semimodule has been introduced in a semiring setting, which
we here call ``strongly projective semimodule''.  In the present
section, we thoroughly investigate such kind of semimodules over
general semirings, semifields and semisimple semirings.

\begin{defn}[{cf.~\cite[Def.~3.1]{ikr:domlzs}}]
  A semimodule $P \in |\mathcal{M}_{S}|$ is \emph{(finitely generated)
    strongly projective} if it is isomorphic to a direct summand of a
  (finitely generated) free right $S$-semimodule.
\end{defn}

\begin{rem}\label{rem:42}
  We note a few easy facts on strongly projective semimodules.
  \begin{enumerate}[label=(\arabic*)]
  \item Any (finitely generated) strongly projective semimodule is a
    (finitely generated) projective semimodule as well, and the
    concepts of ``projectivity'' and ``strong projectivity'' for
    modules over rings coincide.
  \item A strongly projective semimodule $P \in |\mathcal{M}_{S}|$ over
    a zerosumfree semiring~$S$ is zerosumfree as well.
  \item Let $(P_{i})_{i\in I}$ be a family of right $S$-semimodules.
    Then, the right $S$-semimodule $\bigoplus_{i\in I} P_{i}$ is strongly
    projective iff $P_{i}$ is strongly projective for all $i \in I$.
  \end{enumerate}
\end{rem}

The next observation provides a simple criterion for (strong) projectivity.

\begin{lem}\label{lem:43}
  A finitely generated right $S$-semimodule~$P$ is (strongly)
  projective iff there exist a positive integer~$n$ and a (strongly)
  idempotent matrix $A\in M_{n}(S)$ such that $A(S^{n})\cong P$, where
  $A(S^{n})$ is the subsemimodule of the right $S$-semimodule~$S^{n}$
  generated by all column vectors of~$A$.
\end{lem}

\begin{proof}
  If~$P$ is a finitely generated projective right $S$-semimodule,
  then there is some positive integer~$n$ and a homomorphism
  $f \colon S^n \to P$ with a right inverse homomorphism
  $g \colon P \to S^n$, \ie, $f \circ g = \on{id}_P$.  Then $\alpha
  := g \circ f \colon S^n \to S^n$ is an idempotent endomorphism
  with $\alpha(S^n) \cong P$.  If the semimodule~$P$ is in addition
  strongly projective, there exists, for some~$n$ and some
  right $S$-semimodule~$P'$, an isomorphism $S^n \to P \oplus P'$.
  Hence, the corresponding projections are homomorphisms $f \colon S^n
  \to P$, $f' \colon S^n \to P'$ with right inverses $g \colon P \to S^n$,
  $g' \colon P \to S^n$, such that $\alpha := g \circ f$ and $\alpha'
  := g' \circ f'$ are idempotent endomorphisms of~$S^n$ satisfying
  $\alpha + \alpha' = \on{id}_{S^n}$ and $\alpha \circ \alpha' = 0
  = \alpha' \circ \alpha$.  Applying now the standard interpretation
  of endomorphisms of the free right $S$-semimodule~$S^n$ as
  $n \times n$ matrices over~$S$ yields a (strongly) idempotent
  matrix $A \in M_n(S)$ such that $A(S^n) \cong P$.  The converse
  direction is obvious.
\end{proof}

A description of strongly projective semimodules over semifields is
our next goal, for which we need the following useful fact.  Recall
first (see, \eg, \cite[p.~154]{golan:sata}) that a subsemimodule~$K$
of a semimodule $M \in |\mathcal{M}_{S}|$ is said to be \emph{strong}
if $m + m' \in K$ implies $m, m' \in K$ for all $m, m' \in M$.

\begin{prop}\label{prop:44}
  Let $M = \bigoplus_{i\in I} T_i \in |\mathcal{M}_{S}|$ be a direct
  sum of minimal $S$-subsemimodules $T_i \in |\mathcal{M}_{S}|$.
  Then, for every strong subsemimodule $K \subseteq M$, there exists
  a subset $I_K \subseteq I$ such that
  \[ M = K \oplus \big( \textstyle\bigoplus\limits_{i \in I_K} T_i \big)
  \quad \text{ and } \quad
  K \cong \textstyle\bigoplus\limits_{i \in I \setminus I_K} T_i . \]
\end{prop}

\begin{proof}
  By Zorn's lemma, there exists a maximal subset $I_K \subseteq I$
  satisfying the property $K \cap (\bigoplus_{i \in I_K} T_i) = 0$; let
  $N := K + (\bigoplus_{i\in I_K} T_i)$.  We claim that $M = N$.  Indeed,
  for any $i \in I$, as~$T_i$ is minimal, we have either $N \cap T_i = 0$
  or $N \cap T_i = T_i$.  Suppose that $N \cap T_i = 0$ for some $i \in I$.
  Then let $J := I_K \cup \{ i \}$, and for any $x \in K \cap
  (\bigoplus_{i\in J}T_i)$ we can write $x = x_1 + x_2$ with $x_1 \in
  \bigoplus_{i\in I_K}T_i$ and $x_2 \in T_i$.  Since~$K$ is strong, we get
  that $x_1, x_2 \in K$, and thus $x_1\in K \cap (\bigoplus_{i\in I_K}T_i)$
  and $x_2 \in K \cap T_i$.  From the latter one has $x_1 = 0 = x_2$ and
  thus $x = 0$.  Thus, $K\cap (\bigoplus_{i\in J} T_i) = 0$, contradicting
  the maximality of the subset $I_K \subseteq I$.  Therefore,
  $N \cap T_i = T_i$ for all $i \in I$, whence $M = N$.

  Now let $x_1 + x_2 = x_1' + x_2'$, where $x_1, x_1' \in K$
  and $x_2, x_2' \in \bigoplus_{i \in I_K} T_i$.  Using the direct sum
  \[\tag{$\ast$} M \,=\, \big( \textstyle\bigoplus\limits_{i\in I\setminus I_K}
  T_i \big) \oplus \big( \textstyle\bigoplus\limits_{i\in I_K} T_i \big) , \]
  we can write $x_1 = y_1 + y_2$ and $x_1' = y_1' + y_2'$ with
  $y_1, y_1' \in \bigoplus_{i \in I \setminus I_K} T_i$ and
  $y_2, y_2' \in \bigoplus_{i \in I_K} T_i$.  As~$K$ is strong, we
  infer that $y_2, y_2' \in K \cap ( \bigoplus_{i \in I_K} T_i ) = 0$,
  thus $x_1 = y_1$ and $x_1' = y_1'$, so that $y_1 + x_2 = y_1' + x_2'$.
  By $(\ast)$ this implies $x_1 = y_1 = y_1' = x_1'$ and $x_2 = x_2'$,
  whence $M = K \oplus (\bigoplus_{i \in I_K} T_i)$.  The direct sum
  $(\ast)$ also shows, using the corresponding Bourne congruence,
  that $K \cong M / \bigoplus_{i \in I_K} T_i \cong
  \bigoplus_{i \in I \setminus I_K} T_i$.
\end{proof}

\begin{thm}\label{thm:45}
  Every strongly projective right~$D$-semimodule over a division
  semiring~$D$ is free.  In particular, a finitely generated
  semimodule $P \in |\mathcal{M}_{D}|$ is strongly projective if and
  only if there exists a unique nonnegative integer~$n$ such that
  $P \cong D^n$.
\end{thm}

\begin{proof}
  It is clear that~$D$ is either a division ring or a zerosumfree
  division semiring.  Also, the statement is the well-known
  ``classical'' result when~$D$ is a division ring.  So let~$D$ be a
  zerosumfree division semiring and let $P \in |\mathcal{M}_{D}|$ be a
  strongly projective semimodule.  There is a free semimodule
  $F \in |\mathcal{M}_{D}|$, which obviously is zerosumfree, such that
  $F = P \oplus Q$ for some semimodule $Q \in |\mathcal{M}_{D}|$,
  and we claim that~$P$ is a strong subsemimodule of~$F$.  Indeed, let
  $x, y \in F$ such that $x + y \in P$.  Then, $x = p + q$ and
  $y = p' + q'$ for some $p, p'\in P$ and $q, q'\in Q$, hence
  \[ x + y = (p + p') + (q + q') \in P \oplus Q \,. \]
  Therefore, we have that $q +q' = 0$, so that $q =0 = q'$, since~$F$
  is zerosumfree.  This implies that $x = p \in P$ and $y = p' \in P$,
  and thus~$P$ is strong.

  Now noticing that $F = \bigoplus_{i\in I} D_{i}$, where
  $D_{i} \cong D$ for all $i \in I$, is a direct sum of minimal
  right $D$-subsemimodules, by applying Proposition~\ref{prop:44} we get
  that $P \cong \bigoplus_{i\in J} D_{i}$ for some subset $J \subseteq I$,
  whence $P$~is a free semimodule.

  If $P \in |\mathcal{M}_{D}|$ is a finitely generated strongly
  projective semimodule, then there exists a positive integer~$m$ such
  that $D^{m} = P \oplus Q$ for some $Q \in |\mathcal{M}_{D}|$.  By
  the observation above, there exists a nonnegative integer $n \le m$
  such that $P \cong D^n$.  The uniqueness of such a number~$n$ 
    follows from the IBN property of division semirings (see
  \cite[Thm.~5.3]{hebwei:otrosos} or Corollary~\ref{cor:52} below).
\end{proof}

Next we illustrate that the concepts of ``projectivity'' and ``strong
projectivity'' for semimodules, in general, are quite different.  It is
clear that any $\mathbb{B}$-semimodule $M \in |\mathcal{M}_{\mathbb{B}}|$
is an idempotent semimodule and an upper semilattice under the partial
ordering~$\le$ on~$M$ defined for any two elements $x, y \in M$ by
$x \le y$ iff $x + y = y$.  Let us recall the following projectivity
criterion for $\mathbb{B}$-semimodules.

\begin{fac}[{\cite[Thm.~5.3]{hk:tcos}}]\label{fac:46}
  A $\mathbb{B}$-semimodule~$M$ is projective if and only if~$M$ is a
  distributive lattice and $\{ m \in M \mid m \le x \}$ is finite for
  all $x \in M$.
\end{fac}

By applying Theorem~\ref{thm:45} and Fact~\ref{fac:46} it is fairly
easy to provide counterexamples demonstrating the difference of the
concepts of ``projectivity'' and ``strong projectivity'' for general
semimodules.

\begin{exa}\label{exa:47}
  Consider the subsemimodule $P_{\mathbb{B}} := \{ (0, 0), (0, 1), (1, 1) \}$
  of the free semimodule $\mathbb{B}^{2} \in |\mathcal{M}_{\mathbb{B}}|$.
  By Fact~\ref{fac:46}, the semimodule~$P$ is finitely generated projective.
  However, it is obvious that there is no positive integer~$n$ with
  $P \cong \mathbb{B}^{n}$, and therefore, by Theorem~\ref{thm:45},
  $P$ is not a strongly projective $\mathbb{B}$-semimodule.
\end{exa}

Now let us consider strongly projective semimodules under a change
of semirings.  We need the following two functors introduced in
\cite{kat:thcos}.  Given any semirings~$R, S$ and a homomorphism
$\pi \colon R \to S$, every right $S$-semimodule $B_{S}$ may be
considered as a right $R$-semimodule by \emph{pullback} along~$\pi$,
\ie, by defining $b \cdot r := b \cdot \pi (r)$ for any $b \in B$,
$r \in R$.  The resulting $R$-semimodule is written $\pi^{\#} B$, and
it is easy to see that the assignment $B \mapsto \pi^{\#} B$ naturally
constitutes a \emph{restriction} functor $\pi^{\#} \colon \mathcal{M}_{S}
\to \mathcal{M}_{R}$.  The restriction functor $\pi^{\#}$ for left
semimodules is similarly defined.  In particular, the restriction
functor $\pi^{\#} \colon {}_{S} \mathcal{M} \to {}_{R} \mathcal{M}$,
applied to the left $S$-semimodule ${}_{S} S$, gives the
$R$-$S$-bisemimodule $_{R} S_{S} = \pi^{\#} S$.  Then, tensoring
by $\pi^{\#} S$ we have the \emph{extension} functor $\pi_{\#} :=
- \otimes_{R} \pi^{\#} S = - \otimes_{R} S \colon \mathcal{M}_{R} \to
\mathcal{M}_{S}$, which is a left adjoint to the restriction
functor $\pi^{\#} \colon \mathcal{M}_{S} \to \mathcal{M}_{R}$, by
\cite[Prop.~4.1]{kat:thcos}.  Using \cite[Prop.~3.8]{kat:tpaieosoars},
we obtain the following observation, which will prove to be useful.

\begin{prop}\label{prop:48}
  Let $\pi \colon R \to S$ be a semiring homomorphism.
  \begin{enumerate}[label=(\arabic*)]
  \item The extension functor $\pi_{\#} \colon \mathcal{M}_{R} \to
    \mathcal{M}_{S}$ preserves the subcategory of (finitely generated)
    strongly projective semimodules.
  \item The restriction functor $\pi^{\#} \colon \mathcal{M}_{S}
    \to \mathcal{M}_{R}$ preserves the subcategory of (finitely
    generated) strongly projective semimodules if and only if
    $\pi^{\#}(S)$ is a (finitely generated) strongly projective
    right $R$-semimodule.
  \end{enumerate}
\end{prop}

\begin{proof}
  (1) Let~$P$ be a strongly projective $R$-semimodule.  There is
  then a right $R$-semimodule~$Q$ such that $R^{(I)} \cong
  P \oplus Q$ for some basis set~$I$.  Now according to
  \cite[Prop.~3.8]{kat:tpaieosoars}, we obtain that
  \[ S^{(I)} \cong R^{(I)} \otimes_{R} S \cong (P \oplus Q)
  \otimes_{R} S \cong (P \otimes_{R} S) \oplus (Q \otimes_{R} S) , \]
  whence $\pi_{\#}(P) = P \otimes_{R} S$ is a strongly projective right
  $S$-semimodule.

  (2) $(\Longrightarrow)$. It is obvious.

  $(\Longleftarrow)$.  Assume that $\pi^{\#}(S)$ is a strongly
  projective right $R$-semimodule and let~$P$ be a strongly projective
  right $S$-semimodule.  Then,
  \[ R^{(I)} \cong S \oplus A \quad \text{ and } \quad
  S^{(J)}\cong P\oplus B \]%
  for some right $R$-semimodule~$A$ and some right $S$-semimodule~$B$.
  This implies 
  \[ R^{(I\times J)} \cong S^{(J)} \oplus A^{(J)} \cong P \oplus B
  \oplus A^{(J)} \]
  as right $R$-semimodules, thus $\pi^{\#}(P)$ is strongly projective.
\end{proof}

Applying Propositions~\ref{prop:44} and~\ref{prop:48}, the next result
gives a full description of the (finitely generated) strongly
projective semimodules over semisimple semirings.

\begin{thm}\label{thm:49}
  Let~$S$ be a semisimple semiring with direct product representation
  \[ S \cong M_{n_1}(D_{1}) \times \cdots \times M_{n_{r}}(D_{r}) \,, \]
  where $D_{1}, \dots, D_{r}$ are division semirings.  For each
  $1 \le j \le r$ and $1 \le i \le n_{j}$, let~$e_{ii}^{(j)}$ denote
  the $n_{i} \times n_{i}$ matrix units in $M_{n_{j}}(D_{j})$.
  Then, the following holds:
  \begin{enumerate}[label=(\arabic*)]
  \item A right $S$-semimodule is strongly projective if and only if
    it is isomorphic~to 
    \[ (e_{11}^{(1)} S)^{(I_{1})} \oplus \cdots \oplus (e_{11}^{(r)}S)^{(I_{r})} \]
    for some sets $I_{1}, \dots, I_{r}$.
  \item A finitely generated right $S$-semimodule is strongly projective
    if and only if it can be uniquely written in the form
    \[ (e_{11}^{(1)} S)^{k_{1}} \oplus \cdots \oplus (e_{11}^{(r)} S)^{k_{r}} , \]
    where $k_{1}, \dots, k_{r}$ are nonnegative integers.
  \end{enumerate}
\end{thm}

\begin{proof}
  For each $1 \le j \le r$, let $\pi_{j} \colon S \to M_{n_{j}}(D_{j})$
  be the canonical projection.  Then $\pi_{j}^{\#}(M_{n_{j}}(D_{j}))$ is
  obviously a (finitely generated) strongly projective right
  $S$-semimodule and $e_{ii}^{(j)} M_{n_{j}}(D_{j})$ for $1 \le i \le n_{j}$
  is a (finitely generated) strongly projective right $M_{n_{j}}(D_{j})$%
  -semimodule.  Therefore, by Proposition~\ref{prop:48}\,(2),
  $e_{ii}^{(j)} S = \pi_{j}^{\#} (e_{ii}^{(j)} M_{n_{j}}(D_{j}))$ is a
  (finitely generated) strongly projective right $S$-semimodule.
  From this and Remark~\ref{rem:42}\,(3), we immediately see that the
  sufficient conditions of statements~(1) and~(2) are true.

  Assuming that $P \in |\mathcal{M}_{S}|$ is a (finitely generated)
  strongly projective semimodule, we may write it in the form 
  \[ P \,=\, (\pi_{1})_{\#}(P) \oplus \cdots \oplus (\pi_{r})_{\#}(P) \,. \]
  By Proposition~\ref{prop:48}\,(1), $(\pi_{i})_{\#}(P)$ is a
  (finitely generated) strongly right $M_{n_{i}}(D_{i})$-semimodule
  for each $1 \le i \le r$.

  Now we consider the structure of a (finitely generated) strongly
  projective right semimodule over a matrix semiring $M_{m}(D)$ for
  some positive integer~$m$ and division semiring~$D$.  For
  $1 \le i, j \le m$, let $e_{ij}$ be the $m \times m$ matrix units
  in $M_{m}(D)$.  Then each $e_{ii} M_{m}(D)$, for $i = 1, \dots, m$,
  is a minimal right $M_{m}(D)$-semimodule, and $M_{m}(D) = 
  \bigoplus_{i=1}^{m} e_{ii} M_{m}(D)$.  As it is clear that, for all
  $i, j = 1, \dots, m$, the $M_{m}(D)$-semimodules $e_{ii}M_{m}(D)$
  and $e_{jj} M_{m}(D)$ are isomorphic, the semimodules $M_{m}(D)$ and
  $(e_{11}M_{m}(D))^{m}$ are isomorphic as right $M_{m}(D)$-semimodules, too.

  Again, it is easy to see that $D$ is either a division ring or a
  zerosumfree division semiring. In the first scenario, it is a
  well-known ``classical'' result (\eg, \cite[Thm.~1.3.3]{lam:afcinr})
  that every (finitely generated) projective right $M_{m}(D)$-module can
  be uniquely written in the form
  \[ (e_{11} M_{m}(D))^{(I)} \quad \text{ for some (finite) set~$I$} . \]

  Thus from now on, let~$D$ be a zerosumfree division semiring, and
  hence, $M_{m}(D)$ is a zerosumfree semisimple semiring.  Let~$P$ be
  a (finitely generated) strongly projective right
  $M_{m}(D)$-semimodule, \ie, $P$~is a direct summand of a free right
  $M_{m}(D)$-semimodule $M_{m}(D)^{(J)} =: F$ for some (finite)
  set~$J$.  As in the proof of Theorem~\ref{thm:45}, it is easy to show
  that~$P$ is a strong subsemimodule of~$F$.  Also, since $M_{m}(D)
  \cong (e_{11}M_{m}(D))^{m}$, we have that~$F$ is a semisimple
  semimodule and $F \cong (e_{11}M_{m}(D))^{(K)}$ for some set~$K$
  (which can be taken to be finite in case~$P$ is finitely generated).
  From the latter, by Proposition~\ref{prop:44}, there exists a set $I
  \subseteq K$ such that $P \cong (e_{11} M_{m}(D))^{(I)}$ (and~$I$ is
  finite when~$P$ is finitely generated).

  If $(e_{11} M_{m}(D))^{n} \cong (e_{11} M_{m}(D))^{k}$ as right
  $M_{m}(D)$-semimodules, then it easy to see that
  $(e_{11} M_{m}(D))^{n} \cong (e_{11} M_{m}(D))^{k}$ as right
  $D$-semimodules, which means that $D^{mn}\cong D^{mk}$ as right
  $D$-semimodules.  From the latter, by the IBN property of division
  semirings (see \cite[Thm.~5.3]{hebwei:otrosos} or
  Corollary~\ref{cor:52} below), one gets $m n = m k$ and $m = k$.
  Therefore, every finitely generated strongly projective right
  $M_{m}(D)$-semimodule can be uniquely written in the form
  \[ (e_{11} M_{m}(D))^{k} \quad \text{ for some nonnegative integer~}k . \]

  Finally, we notice that $e_{11}^{(j)} S \cong \pi_{j}^{\#} (e_{11}^{(j)}
  M_{n_{j}}(D_{j}))$ as right $S$-semimodules for all $j = 1, \dots, r$,
  and conclude that the necessary conditions of statements~(1) and~(2)
  are true as well.
\end{proof}

At the end of this section we establish some preparatory results
regarding strongly projective semimodules over semirings that are
direct limits of directed families of semirings.  Let us recall a few
general notions from universal algebra (see, \eg, \cite[Ch.~3]%
{gratzer:ua}) in a semiring context.  A partially ordered set~$I$ is
\emph{directed} if any two elements of~$I$ have an upper bound in~$I$.
Denoting by $\mathcal{SR}$ the category of semirings, a \emph{direct
  system} $\{ S_{i} \mid \phi_{ij} \}$ of semirings over a directed
set~$I$ consists of a family $\{ S_{i} \}_{i \in I}$ of semirings
$S_{i} \in |\mathcal{SR}|$, together with semiring homomorphisms
$\phi_{ij} \colon S_{i} \to S_{j}$ for $i \le j$, such that, for all
$i, j, k \in I$, if $i \le j \le k$, then $\phi_{jk} \phi_{ij} = \phi_{ik}$
and $\phi_{ii} = \on{id}_{S_{i}}$.  If one defines a binary relation
``$\equiv$'' on the disjoint union $\bigcup_{i \in I} S_{i}$ of the
sets~$S_{i}$ by $x \equiv y$ iff $x \in S_{i}$, $y \in S_{j}$ for some
$i, j \in I$, and there exists $z \in S_{k}$ such that $i, j \le k$
and $\phi_{ik}(x) = z = \phi_{jk}(y)$, then this is easily seen to be an
equivalence relation.  Considering its set $S := \{ [x] \mid x \in
\bigcup_{i \in I} S_{i} \}$ of equivalence classes, it is not hard to
verify that by defining
\[ [x] + [y] := [ \phi_{ik}(x) + \phi_{jk}(y) ] \quad \text{ and }
\quad [x] \cdot [y] := [ \phi_{ik}(x) \cdot \phi_{jk}(y) ] \,, \]
where $x \in S_{i}$, $y \in S_{j}$ and $i, j \le k$, one obtains a
semiring $\varinjlim_{I} S_{i} := S = (S, +, \cdot, [0], [1])$ called
the \emph{direct limit} of the direct system $\{ S_{i} \mid \phi_{ij} \}$
of semirings.  It is also easy that there is a family
$\{ \phi_{i} \}_{i \in I}$ of canonical homomorphisms $\phi_{i} \colon
S_{i} \to S$ defined by $\phi_{i}(x) := [x]$ for any $x \in S_{i}$,
so that $\phi_{i} = \phi_{j} \phi_{ij}$ for all $i \le j$; and if
all~$\phi_{ij}$ for $i \le j$ are embeddings, then all $\phi_{i}$,
$i \in I$, are embeddings, too.

Our next result, needed in a sequel, is of a ``technical'' nature and
can be justified by using Lemma~\ref{lem:43} and repeating verbatim
the proof of \cite[Lemma 15.10]{g:vnrr} in the ring setting.  However,
for the reader's convenience, we briefly sketch here an alternative,
more homological, proof based on the tensor product construction and
related results considered in~\cite{kat:tpaieosoars} and~\cite{kat:thcos}.

\begin{prop}[{cf.~\cite[Lem.~15.10]{g:vnrr}}]\label{prop:410}
  Let~$S$ be a direct limit of a direct system $\{ S_{i} \mid \phi_{ij} \}$
  of semirings.  Then the following statements are true:
  \begin{enumerate}[label=(\arabic*)]
  \item If~$P$ is a finitely generated (strongly) projective right
    $S$-semimodule, then, for some~$m$, there exists a finitely
    generated (strongly) projective right $S_{m}$-semimodule~$Q$ such
    that $Q \otimes_{S_{m}} S \cong P$.
  \item If $P \otimes_{S_{i}} S \cong Q \otimes_{S_{i}} S$ for some~$i$
    and finitely generated (strongly) projective right $S_{i}$-semimodules
    $P, Q$, then $P \otimes_{S_{i}} S_{k} \cong Q \otimes_{S_{i}} S_{k}$ for
    some $k \ge i$.
  \end{enumerate}
\end{prop}

\begin{proof}
  (1) Since the semimodule $P \in |\mathcal{M}_{S}|$ is a finitely
  generated summand of a free $S$-semimodule, we can consider all
  components of a finite generator~$P_0$ to be elements of some
  semiring~$S_{m}$, and let $Q := P_0 S_{m} \in |\mathcal{M}_{S_{m}}|$.
  It is easy to see that~$Q$ is a (strongly) projective $S_{m}$-semimodule
  and $\varinjlim_{I} S S_{i} = {}_{S}S$ in the category $_{S}\mathcal{M}$.
  Then, $P = P_0 S = P_0 S \otimes_{S} S = P_0 S \otimes_{S} \varinjlim_{I}
  S S_{i} = \varinjlim_{I}(P_0 S \otimes_{S} S S_{i}) = \varinjlim_{I}(P_0 S_{m}
  \otimes_{S_{m}} S_{S} \otimes_{S} S S_{i}) = \varinjlim_{I}(P_0 S_{m}
  \otimes_{S_{m}} S S_{i}) = P_0 S_{m} \otimes_{S_{m}} \varinjlim_{I} S S_{i}
  = Q \otimes_{S_{m}} \varinjlim_{I} S S_{i} = Q \otimes_{S_{m}} S$.

  (2) Since semimodules $P, Q \in |\mathcal{M}_{S_{i}}|$ are finitely
  generated summands of free $S_{i}$-semimodules, the semimodules
  $P \otimes _{S_{i}} S, \, Q \otimes_{S_{i}} S \in |\mathcal{M}_{S}|$
  are finetely generated summands of free $S$-semimodules as well.
  Since any isomorphism between $S$-semimodules $P\otimes_{S_{i}} S$
  and $Q \otimes_{S_{i}} S$ is defined by the finite number of their
  generators and a finite number of elements of~$S$, and taking
  into consideration the nature of the congruence relation in the
  construction of the tensor product, we can consider that all
  elements of~$S$ involved into the isomorphism $P \otimes_{S_{i}} S
  \cong Q \otimes _{S_{i}}S$ are elements of some semiring~$S_{k}$
  with $k \ge i$.  
  Therefore, $P \otimes_{S_{i}} S \cong P \otimes_{S_{i}}(S_{k} \otimes
  _{S_{k}} S) \cong (P \otimes_{S_{i}} S_{k}) \otimes_{S_{k}}S$ and $Q \otimes
  _{S_{i}} S \cong Q \otimes_{S_{i}}(S_{k} \otimes_{S_{k}} S)\cong (P \otimes
  _{S_{i}} S_{k}) \otimes_{S_{k}} S$, and it is clear that $P \otimes
  _{S_{i}} S_{k}$ $\cong Q \otimes_{S_{i}} S_{k}$.
\end{proof}

\section{Characterizing ultramatricial algebras
  by monoids of isomorphism classes of projective
  semimodules}\label{sec:5}

In this section, we introduce the monoids $\mathcal{V}(S)$ and
$\mathcal{SV}(S)$ of isomorphism classes of finitely generated
projective and strongly projective, respectively, semimodules over a
semiring~$S$ and demonstrate their roles in the characterization of
the class of ultramatricial algebras over a semifield.  The proof of
the main result is essentially based on the presentation in
\cite[Ch.~15]{g:vnrr}.

From now on, let $\mathcal{V}(S)$ be the set of isomorphism classes of
finitely generated projective right $S$-semimodules and, for a finitely
generated projective $S$-semimodule $P\in |\mathcal{M}_{S}|$ let
$\overline{P} \in \mathcal{V}(S)$ denote the class of finitely
generated projective right $S$-semimodules isomorphic to~$P$.
Furthermore, defining for any isomorphism classes $\overline{P}$ and
$\overline{Q}$ an addition ``$+$'' by $\overline{P} + \overline{Q}
:= \overline{P \oplus Q}$, it is easy to see that the set $\mathcal{V}(S)$
becomes a commutative monoid $(\mathcal{V}(S), +, \overline{0})$ with the
zero element~$\overline{0}$.  The monoid $\mathcal{V}(S)$ is always a
\emph{zerosumfree} monoid (or a \emph{strict cone} in the terminology of
\cite[p.~202]{g:vnrr}), \ie, $x + y = 0$ implies $x = y = 0$.

Let $\mathcal{SV}(S) \subseteq \mathcal{V}(S)$ denote the submonoid of
the monoid $\mathcal{V}(S)$ consisting of all classes $\overline{P}$
with a finitely generated strongly projective $S$-semimodule
$P \in |\mathcal{M}_{S}|$, \ie, $\mathcal{SV}(S) := \{ \overline{P}
\in \mathcal{V}(S) \mid P \text{ is a strongly projective
  $S$-semimodule} \}$.

Before presenting ``computational'' examples of the monoids
$\mathcal{V}(S)$ and $\mathcal{SV}(S)$, we start with some useful
observations.  Recall that a semiring~$S$ has the IBN property if
$S^m \cong S^n$ (in $\mathcal{M}_S$) implies $m = n$, for any natural
numbers $m, n$.

\begin{lem}\label{lem:51}
  If $\phi \colon S \to T$ is a semiring homomorphism and~$T$ has the
  IBN property, then~$S$ has it as well.
\end{lem}

\begin{proof}
  Indeed, if $S^{m} \cong S^{n}$ in $\mathcal{M}_{S}$, then by
  \cite[Prop.~3.8]{kat:tpaieosoars} and \cite[Thm.~3.3]%
  {kat:thcos}, one readily has $T^{m} \cong S^{m} \otimes_{S}
  T\cong S^{n} \otimes_{S} T\cong T^{n}$ in $\mathcal{M}_{T}$,
  whence $m = n$.
\end{proof}

\begin{cor}\label{cor:52}
  Division semirings and commutative semirings satisfy IBN.
\end{cor}

\begin{proof}
  The case when~$S$ is a division semiring was justified in
  \cite[Thm.~5.3]{hebwei:otrosos}; alternatively, one can use
  Lemma~\ref{lem:51} and the fact that for any zerosumfree division
  semiring~$S$ there is a semiring homomorphism $\phi \colon S \to
  \mathbb{B}$ into the finite IBN semiring $\mathbb B$, given by
  $\phi(0) = 0$ and $\phi(s) = 1$ for $0 \ne s \in S$.

  If~$S$ is a commutative semiring, by Zorn's lemma there exists a
  maximal congruence $\rho $ on $S$, so that $T := S/\rho$ is a
  congruence-simple commutative semiring (\ie, $T$ has only the
  trivial congruences).  By \cite[Thm.~10.1]{bshhurtjankepka:scs}, $T$
  is either a field or the Boolean semifield~$\mathbb{B}$, and
  hence~$T$ has the IBN property.  From Lemma~\ref{lem:51} it follows
  that~$S$ has the IBN property, too.
\end{proof}

\begin{exas}\label{exa:53}
  Cases of semirings~$S$, for which the monoids $\mathcal{V}(S)$ and
  $\mathcal{SV}(S)$ are essentially known, include the following.
  \begin{enumerate}[label=(\arabic*)]
  \item By Theorem~\ref{thm:45}, we have $\mathcal{SV}(D) \cong
    \mathbb{Z}^{+}$ for any division semiring~$D$.  In particular,
    $\mathcal{SV} (\mathbb{B}) \cong \mathbb{Z}^{+}$.  However,
    Example~\ref{exa:47} shows that $\mathcal{SV}(\mathbb{B})$ is a
    proper submonoid of $\mathcal{V}(\mathbb{B})$.  In fact, the
    monoid $\mathcal{V}(\mathbb{B})$ contains a free commutative monoid
    with countable basis. 
  \item For any semiring~$S$, obviously $\mathcal{V}(S) = \mathcal{SV}(S)$
    iff all finitely generated projective right $S$-semimodules are
    strongly projective.  In particular, for any division semiring~$D$,
    we have $\mathcal{V}(D) = \mathcal{SV}(D)$ iff~$D$ is \emph{weakly
      cancellative}, \ie, $a + a = a + b$ implies $a = b$, for all
    $a, b \in S$ (\cite[p.~4026]{ik:ospopsops}).  Indeed, by
    Theorem~\ref{thm:45}, one has $\mathcal{V}(D) = \mathcal{SV}(D)$
    iff every finitely generated projective right $D$-semimodule is
    free, that is, by \cite[Prop.~3.1, Thm.~3.2]{ik:ospopsops},
    iff~$D$ is a weakly cancellative division semiring.
  \item Let~$S$ be a semiring for which every finitely generated
    projective right~$S$-semimodule is free.  (For example, in
    \cite{patch:psoswvini} and \cite{ik:ospopsops}, polynomial
    semirings are considered having this property.)  Then the monoids
    $\mathcal{V}(S)$ and $\mathcal{SV}(S)$ are cyclic monoids
    generated by the element~$\overline{S}$.  If in addition~$S$ has
    the IBN property (\eg, if~$S$ is a commutative semiring), then
    the monoids $\mathcal{V}(S)$ and $\mathcal{SV}(S)$ are exactly
    $\mathbb{Z}^{+}$.
  \end{enumerate}
\end{exas}

Notice that if $\phi \colon R \to S$ is a semiring homomorphism, then,
taking into account Proposition~\ref{prop:48}, we see that~$\phi$ induces
a well-defined monoid homomorphism $\mathcal{V}(\phi) \colon \mathcal{V}(R)
\to \mathcal{V}(S)$ such that $\mathcal{V}(\phi)(\overline{P}) =
\overline{\phi_{\#}(P)} = \overline{P \otimes_R S}$; furthermore,
it holds that $\mathcal{V}(\phi)(\mathcal{SV}(R)) \subseteq
\mathcal{SV}(S)$.  From these observations, it is routine to check
that $\mathcal{V}$ and $\mathcal{SV}$ give covariant functors from the
category of semirings to the category of commutative monoids.

Recall that a commutative monoid~$M$ is \emph{conical} if $x + y = 0$
implies $x = 0 = y$, for any $x, y \in M$.  An \emph{order-unit} in
the monoid~$M$ is an element~$u$ in~$M$ such that for every $x \in M$
there exist $y \in M$ and a positive integer~$n$ such that $x + y = nu$.
Throughout this section, we denote by~$\mathcal{C}$ the category
consisting of all pairs $(M, u)$, where~$M$ is a conical monoid
and~$u$ is an order-unit in~$M$, with morphisms from an object
$(M, u)$ to an object $(M', u')$ to be the monoid homomorphisms
$f \colon M \to M'$ satisfying $f(u) = u'$.

For any semiring~$S$, observe that~$\overline{S}$ is an order-unit in
the conical monoid $\mathcal{SV}(S)$.  (Note that~$\overline{S}$ is
no order-unit in the monoid $\mathcal{V}(S)$, unless $\mathcal{V}(S)
= \mathcal{SV}(S)$.)  Therefore, we have an object $(\mathcal{SV}(S),
\overline{S})$ in the category~$\mathcal{C}$ defined above.  Given
any semiring homomorphism $\phi \colon R \to S$, note that
$\mathcal{SV}(\phi)$ maps $\overline{R}$ to $\overline{S}$, so
that $\mathcal{SV}(\phi)$ is a morphism in $\mathcal{C}$ from
$(\mathcal{SV}(R), \overline{R})$ to $(\mathcal{SV}(S), \overline{S})$.
Thus, $(\mathcal{SV}(-), \overline{-})$ defines a covariant functor
from the category of semirings to the category~$\mathcal{C}$.

In order to apply $\mathcal{SV}$ to direct limits and finite products
of semirings, we consider direct limits and finite products in the
category~$\mathcal{C}$.  Given a direct system of objects $(M_i, u_i)$
and morphisms~$f_{ij}$ in~$\mathcal{C}$, we first form the direct
limit~$M$ of the commutative monoids~$M_i$ and let $f_i \colon M_i \to M$
denote the canonical homomorphisms.  One can easily check that~$M$ is a
conical monoid.  Since $f_{ij}(u_i) = u_j$ whenever $i \le j$, there is a
unique element $u \in M$ such that $f_i(u_i) = u$ for all~$i$, and we
observe that~$u$ is an order-unit in~$M$.  Thus, $(M, u)$ is an object in
$\mathcal{C}$, and each~$f_i$ is a morphism from $(M_i, u_i)$ to $(M, u)$.
It is easy to see that $(M, u)$ is the direct limit of the $(M_i, u_i)$.

It is standard how to form finite products in the category~$\mathcal{C}$.
Namely, given objects $(M_1, u_1), \dots, (M_n, u_n)$ in $\mathcal{C}$,
we set $M = M_1 \times \cdots \times M_n$, which is a conical monoid,
together with $u = (u_1, \dots, u_n)$, which is an order-unit in~$M$.
It is easy to check that $(M, u)$ is the product of the $(M_i, u_i)$
in~$\mathcal{C}$.

\begin{prop}\label{prop:54}
  The functor $(\mathcal{SV}(-), \overline{-}) \colon \mathcal{SR}
  \to \mathcal{C}$ preserves direct limits and finite products.
\end{prop}

\begin{proof}
  We adapt the proof of~\cite[Prop.~15.11]{g:vnrr} to our situation.
  Let~$S$ be the direct limit of a direct system $\{ S_{i} \mid
  \phi_{ij} \}$ of semirings, and for each~$i$ let $\phi_{i} \colon
  S_{i} \to S$ be the canonical homomorphism.  Let $(M, u)$ be the
  direct limit of the monoids $(\mathcal{SV}(S_{i}), \overline{S_{i}})$
  in~$\mathcal{C}$, and for each~$i$ let $f_{i} \colon (\mathcal{SV}(S_{i}),
  \overline{S_{i}}) \to (M,u)$ be the canonical homomorphism.  We have
  morphisms $\mathcal{SV}(\phi_{i}) \colon (\mathcal{SV}(S_{i}),
  \overline{S_{i}}) \to (\mathcal{SV}(S), \overline{S})$ such that
  $\mathcal{SV}(\phi_{j}) \mathcal{SV}(\phi_{ij}) = \mathcal{SV}(\phi_{i})$
  whenever $i \le j$; hence, there exists a unique monoid homomorphism
  $g \colon (M,u) \to (\mathcal{SV}(S), \overline{S})$ such that
  $g f_{i} = \mathcal{SV}(\phi _{i})$ for all~$i$.  We are going to
  prove that~$g$ is an isomorphism.

  Given $\overline{P} \in \mathcal{SV}(S)$, we see from 
  Proposition~\ref{prop:410}\,(1) that there is a finitely generated
  strongly projective right $S_{i}$-semimodule~$Q$ for some~$i$ with
  $Q \otimes_{S_{i}} S \cong P$.  Then $\overline{Q} \in
  \mathcal{SV}(S_{i})$, and so $f_{i}(\overline{Q})\in M$, and also
  \[ g f_{i}(\overline{Q}) = \mathcal{SV}(\phi_{i})(\overline{Q})
  = \overline{Q \otimes_{S_{i}} S} = \overline{P} . \]%
  This implies that~$g$ is surjective.

  Now, let $x, y \in M$ be such that $g(x) = g(y)$.  Then there exist~$i$
  and~$j$ such that $f_{i}(\overline{P}) = x$ and $f_{j}(\overline{Q}) = y$
  for some $\overline{P} \in \mathcal{SV}(S_{i})$ and $\overline{Q}
  \in \mathcal{SV}(S_{j})$.  Choosing~$k$ such that $i, j \le k$, we have
  that \[ \mathcal{SV}(\phi_{k}) \mathcal{SV}(\phi_{ik})(\overline{P})
  = g f_{i}(\overline{P}) = g(x) = g(y) = g f_{j}(\overline{Q})
  = \mathcal{SV}(\phi_{k}) \mathcal{SV}(\phi_{jk})(\overline{Q}) \,, \]
  which means $\mathcal{SV}(\phi_{ik})(\overline{P}) \otimes_{S_{k}} S
  = \mathcal{SV}(\phi_{jk})(\overline{Q}) \otimes_{S_{k}} S$.  
  By Proposition~\ref{prop:410}\,(2),
  \[ \mathcal{SV}(\phi_{ik})(\overline{P}) \otimes_{S_{k}} S_{t}
  = \mathcal{SV}(\phi_{jk})(\overline{Q}) \otimes_{S_{k}} S_{t} \]
  for some $t \ge k$.  Then, $\mathcal{SV}(\phi_{it})(\overline{P})
  = \mathcal{SV}(\phi_{kt}) \mathcal{SV}(\phi _{ik})(\overline{P})
  = \mathcal{SV}(\phi_{ik}) (\overline{P}) \otimes_{S_{k}} S_{t}
  = \mathcal{SV}(\phi_{jk}) (\overline{Q}) \otimes_{S_{k}} S_{t}
  = \mathcal{SV}(\phi_{kt}) \mathcal{SV} (\phi_{jk}) (\overline{Q})
  = \mathcal{SV}(\phi_{jt}) (\overline{Q})$, and hence, 
  \[ x = f_{i}(\overline{P}) = f_{t} \mathcal{SV}(\phi _{it})
  (\overline{P}) = f_{t} \mathcal{SV}(\phi_{jt}) (\overline{Q})
  = f_{j}(\overline{Q}) = y \,. \]
  Thus~$g$ is injective.

  Using the same argument above and Proposition~\ref{prop:48}, we get
  that the functor $(\mathcal{SV}(-), \overline{-})$ preserves finite
  products.
\end{proof}

The following fact shows that every free commutative monoid of finite
rank occurs as a monoid of isomorphism classes of strongly projective
semimodules of a zerosumfree semisimple semiring.

\begin{prop}\label{prop:55}
  Let~$S$ be a semisimple semiring with direct product representation
  \[ S \,\cong\, M_{n_{1}}(D_{1}) \times \cdots \times M_{n_{r}}(D_{r}) \,, \]
  where $D_{1}, \dots, D_{r}$ are division semirings.  For each
  $1 \le j \le r$ and $1 \le i \le n_{j}$, let~$e_{ii}^{(j)}$ be the
  $n_{i}\times n_{i}$ matrix units in $M_{n_{j}}(D_{j})$.  Then,
  $\mathcal{SV}(S)$ is a free commutative monoid with basis
  $\{ \overline{{e_{11}}^{(1)} S}, \dots, \overline{{e_{11}}^{(r)} S} \}$,
  and $(\mathcal{SV}(S), \overline{S}) \cong \big( (\mathbb{Z}^{+})^{r},
  (n_{1}, \dots, n_{r}) \big)$.
\end{prop}

\begin{proof}
  According to Proposition~\ref{prop:54}, we have 
  \[ \mathcal{SV}(S) \,\cong\, \mathcal{SV}(M_{n_{1}}(D_{1}))
  \oplus \cdots \oplus \mathcal{SV}(M_{n_{r}}(D_{r})) \,. \]
  Moreover, for each $1 \le j \le r$, the monoid $\mathcal{SV}%
  (M_{n_{j}}(D_{j}))$ is a free commutative monoid with basis
  $\{ \overline{{e_{11}}^{(j)} S} \}$, by Theorem~\ref{thm:49}\,(2).
  Using those observations, we immediately get the statement.
\end{proof}

From Corollary~\ref{cor:35} and Proposition~\ref{prop:55} we readily
see that every free commutative monoid of finite rank appears as a
monoid $\mathcal{SV}(S)$ for some additively idempotent
congruence-semisimple semiring~$S$.  Motivated by this remark and the
Realization Problem, which constitutes a very active area in
non-stable $K$-theory (we refer the reader to \cite{ag:trpfswmatap}
and the references given there for a recent progress on the
Realization Problem), it is natural to pose the following problem.

\begin{prob}
  Describe commutative monoids which can be realized as either a
  monoid $\mathcal{SV}(S)$ or $\mathcal{V}(S)$ for an additively
  idempotent semiring~$S$.
\end{prob}

Now we define a central notion for the present article, which has been
investigated in Section~3 for a special case.

\begin{defn}
  Let~$F$ be a semifield.
  \begin{enumerate}[label=(\arabic*)]
  \item A \emph{matricial $F$-algebra} is an $F$-algebra isomorphic to
    $M_{n_{1}}(F) \times \cdots \times M_{n_{r}}(F)$, for some positive
    integers $n_{1}, \dots, n_{r}$.
  \item An $F$-algebra is said to be \emph{ultramatricial} if it is
    isomorphic to the direct limit (in the category of unital $F$-algebras)
    of a sequence $S_{1} \to S_{2} \to \cdots$ of matricial $F$-algebras.
  \end{enumerate}
\end{defn}

We note the following simple observation, pointing out that its
justification differs significantly from the arguments in the
``classical'' ring case.

\begin{prop}\label{prop:57}
  An algebra~$S$ over a semifield~$F$ is ultramatricial if and only
  if~$S$ is the union of an ascending sequence $S_{1} \subseteq S_{2}
  \subseteq \cdots$ of matricial subalgebras.
\end{prop}

\begin{proof}
  Suppose that~$S$ is the direct limit of a sequence $S_1 \to S_2 \to \dots$
  of matricial $F$-algebras with canonical homomorphisms $\phi_i \colon
  S_i \to S$, then~$S$ is the union of the ascending sequence
  $\phi_1(S_1) \subseteq \phi_2(S_2) \subseteq \cdots$.  Therefore, it
  is left to show that each $\phi_i(S_i)$ is a matricial $F$-algebra.
  
  This follows, as is easy to verify, from the following claim.  If~$R$
  and~$S$ are matricial $F$-algebras, where $R = R_1 \times \dots
  \times R_r$ and $R_i = M_{n_i}(F)$ for some positive integers~$n_i$,
  and if $\phi \colon R \to S$ is any algebra homomorphism, then $\phi(R)
  \cong \prod_{j \in J} R_j$ for some subset $J \subseteq \{ 1, \dots, r \}$;
  thus $\phi(R)$ is also a matricial $F$-algebra.

  To prove this claim, note that $\phi \colon R \to S$ as above induces an
  isomorphism $R / \ker(\phi) \to \phi(R)$, $\overline{r} \mapsto \phi(r)$,
  where $\ker(\phi) = \{ (x, y) \in R \times R \mid \phi(x) = \phi(y) \}$
  is its kernel congruence.  Using the congruences $\rho_{i}
  := \ker(\phi) \,\cap\, R_i \!\times\! R_i$, we have
  \[ \phi(R) \,\cong\, R / \ker(\phi) \,\cong\, R_1 / \rho_{1}
  \times \cdots \times R_r / \rho_{r} \,. \]
  We argue that each~$\rho_i$ is a trivial congruence on $R_i$, which
  proves the claim.  Now the semifield~$F$ is either a field or a
  zerosumfree semifield.  If~$F$ is a field, then $\rho_i$ is
  obviously a trivial congruence, since $R_i$ is a simple ring.

  Assume then that~$F$ is a zerosumfree semifield, and that~$\rho_{i}$
  is not the identity one.  There are distinct elements $A = (a_{jk})$
  and $B = (b_{jk})$ in $R_i$ such that $A \,\rho_i\, B$, thus $a := a_{jk}
  \ne b_{jk} =: b$ for some $j, k \in \{ 1, \dots, n_i \}$.  Denoting by
  $E_{jk}$ the matrix units in $R_i$ and using $E_{jk} E_{kt} = E_{jt}$,
  we readily infer that $a E_{jk} \,\rho_i\, b E_{jk}$ for all $j, k$,
  whence $a I_{n_i} \,\rho_i\, b I_{n_i}$.  This implies that 
  \[ a \phi(I_{n_i}) = \phi(a I_{n_i}) = \phi(b I_{n_i}) = b \phi(I_{n_i}) \,. \]
  If $\phi(I_{n_i}) \ne 0$, then since~$S$ is a matricial $F$-algebra
  we must have that $a = b$, giving a contradiction. 
  Therefore, $\phi(I_{n_i}) = 0$, \ie, $I_{n_i} \,\rho_i\, 0$, which
  implies $C = C I_{n_i} \,\rho_i\, 0 I_{n_i} = 0$ for all $C \in R_i$,
  so that $\rho_i = R_i \times R_i$ is the universal one.
\end{proof}

The subsequent fact, which immediately follows from
Propositions~\ref{prop:54} and~\ref{prop:55}, provides some
information on the monoid $\mathcal{SV}(S)$ for an ultramatricial
$F$-algebra over a semifield~$F$.

\begin{rem}\label{rem:58}
  Let~$S$ be a direct limit of a sequence $S_1 \to S_2 \to \dots$ of
  matricial $F$-algebras with canonical homomorphisms $\phi_i \colon
  S_i \to S$, and suppose that $S_i = \smash{M_{n_1^i}(F) \times \dots
    \times M_{n_{r(i)}^i}(F)}$ for some positive integers $n_1^i, \dots,
  n_{r(i)}^i$, for each~$i$, denoting by ${e_{11}}^{(j,i)} \in S_i$ the
  matrix units in $\smash{M_{n_j^i}(F)}$.  Then $\mathcal{SV}(S)$ is
  a cancellative monoid generated by the set $\{
  \overline{\phi_i({e_{11}}^{(j,i)}) S} \mid 1 \le j \le r(i) ,\,
  i = 1, 2, \dots \}$.
\end{rem}

In order to establish the main results of this section, we state the
following useful lemma, which proof is essentially based on the one of
\cite[Lem.~15.23]{g:vnrr}.

\begin{lem}\label{lem:59}
  Let~$F$ be a semifield, let~$S$ be a matricial $F$-algebra, and
  let~$T$ be any unital $F$-algebra.
  \begin{enumerate}[label=(\arabic*)]
  \item For any morphism $f \colon (\mathcal{SV}(S), \overline{S})
    \to (\mathcal{SV}(T), \overline{T})$ in the category~$\mathcal{C}$,
    there exists an $F$-algebra homomorphism $\phi \colon S \to T$
    such that $\mathcal{SV}(\phi) = f$.
  \item Let $\phi, \psi \colon S \to T$ be $F$-algebra homomorphisms.
    If $\mathcal{SV}(\phi) = \mathcal{SV}(\psi)$, then there exists an
    inner automorphism $\theta$ of~$S$ such that $\phi = \theta \psi$.
    Moreover, if in addition~$T$ is an ultramatricial $F$-algebra, then
    $\mathcal{SV}(\phi) = \mathcal{SV}(\psi)$ if and only if there exists
    an inner automorphism~$\theta$ of~$S$ such that $\phi = \theta \psi$.
  \end{enumerate}
\end{lem}

\begin{proof}
  There are orthogonal central idempotents $e_{1}, \dots, e_{r} \in S$
  with $e_{1} + \ldots + e_{r} = 1$ and each $e_{i} S \cong M_{n_i}(F)$
  for some positive integer $n_i$.  For each~$i$, denoting by
  $\smash{e_{jk}^{(i)}} \in e_{i} S$ the matrix units, we have that
  $e_{11}^{(i)} + \ldots + e_{n_i n_i}^{(i)} = e_{i}$.  According to
  Proposition~5.5, $\mathcal{SV}(S)$ is a free commutative monoid with
  basis $\{\overline{{e_{11}}^{(1)} S}, \dots, \overline{{e_{11}}^{(r)} S}\}$.

  (1) For each~$i$, we have $\overline{e_{i} S} \in \mathcal{SV}(S)$ and
  so $f(\overline{e_{i}S}) \in \mathcal{SV}(T)$, \ie, $f(\overline{e_{i} S})
  = \overline{P_{i}}$ for some finitely generated strongly projective right
  $T$-semimodule $P_{i}$.  Since
  \[ \overline{P_{1} \oplus \ldots \oplus P_{r}} = \overline{P_{1}} + \ldots
  + \overline{P_{r}} = f(\smsum_{i=1}^{r} \overline{e_{i} S}) = f(\overline{S})
  = \overline{T} \,, \]
  we have $P_{1} \oplus \cdots \oplus P_{r} \cong T$ as right $T$-semimodules.
  Consequently, there exist orthogonal idempotents $g_{1},\dots,g_{n}\in T$ such
  that $g_{1}+\cdots +g_{n}=1$ and each $g_{i}T\cong P_{i}$. Note that each $%
  \overline{g_{i}T}=\overline{P_{i}}=f(\overline{e_{i}S})$.

  Furthermore, for each~$i$, we have $f(\overline{{e_{11}}^{(i)} S}) =
  \overline{Q_{i}}$ for some finitely generated strongly projective
  right $T$-semimodule $Q_{i}$.  Because
  \[ \overline{Q_{i}^{n_i}} = f(n_i \, \overline{{e_{11}}^{(i)} S})
  = f(\overline{{e_{11}}^{(i)} S} + \ldots + \overline{{e_{n_i n_i}}^{(i)} S})
  = f(\overline{e_{i} S}) = \overline{g_{i} T} \,, \]
  we have $Q_{i}^{n_i} \cong g_{i}T$.  As a result, there exist $n_i \times n_i$
  matrix units $\smash{g_{jk}^{(i)}} \in g_{i}Tg_{i}$ such that $g_{11}^{(i)} T \cong
  Q_{i}$ and $g_{11}^{(i)} + \ldots + g_{n_i n_i}^{(i)} = g_{i}$.  Note that 
  $\overline{{g_{11}}^{(i)} T} = \overline{Q_{i}} = f(\overline{{e_{11}}^{(i)}
    S})$.  For every~$i$, there is a unique $F$-algebra homomorphism from 
  $e_{i}S$ into $g_{i} T g_{i}$ sending $e_{jk}^{(i)}$ to $g_{jk}^{(i)}$, for all 
  $j, k = 1, \dots, n_i$.  Consequently, there is a unique $F$-algebra
  homomorphism $\phi \colon S \to T$ such that $\smash{\phi(e_{jk}^{(i)})
    = g_{jk}^{(i)}}$ for all $i, j, k$.  Then 
  \[ \mathcal{SV}(\phi) (\overline{{e_{11}}^{(i)} S}) 
  = \overline{\phi({e_{11}}^{(i)}) T} = \overline{{g_{11}}^{(i)} T }
  = f(\overline{{e_{11}}^{(i)} S}) \]%
  for all $i = 1, \dots, r$.  \sloppy Since $\mathcal{SV}(S)$ is a free 
  commutative monoid with basis $\{ \overline{{e_{11}}^{(1)} S}, \dots, 
  \overline{{e_{11}}^{(r)} S} \}$, we conclude that $\mathcal{SV}(\phi) = f$.

  (2) Assume that $\mathcal{SV}(\phi) = \mathcal{SV}(\psi)$.  Set 
  $g_{jk}^{(i)} = \phi(e_{jk}^{(i)})$ and $h_{jk}^{(i)} = \psi(e_{jk}^{(i)})$
  for all $i, j, k$.  Note that the $\smash{g_{jj}^{(i)}}$ are pairwise
  orthogonal idempotents in~$T$ such that $\sum_{i=1}^{r} \sum_{j=1}^{n_i}
  g_{jj}^{(i)} = 1$, and similarly for the $h_{jj}^{(i)}$.  For $i = 1, 
  \dots, r$, we have \[ \overline{g_{11}^{(i)} T}
  = \mathcal{SV}(\phi) (\overline{e_{11}^{(i)} S})
  = \mathcal{SV}(\psi) (\overline{e_{11}^{(i)} S})
  = \overline{h_{11}^{(i)} T}; \] hence, $g_{11}^{(i)}T\cong h_{11}^{(i)}T$.
  Consequently, there are elements $x_{i}\in g_{11}^{(i)}Th_{11}^{(i)}$ and
  $y_{i}\in h_{11}^{(i)}Tg_{11}^{(i)}$ such that $x_{i}y_{i}=g_{11}^{(i)}$ and
  $y_{i}x_{i}=h_{11}^{(i)}$.

  Set $x = \sum_{i=1}^{r}\sum_{j=1}^{n_i}g_{j1}^{(i)}x_{i}h_{1j}^{(i)}$ and 
  $y = \sum_{i=1}^{r}\sum_{j=1}^{n_i}h_{j1}^{(i)}y_{i}g_{1j}^{(i)}$.  We then have 
  \begin{align*}
    xy &= \smsum_{i,k=1}^{r} \smsum_{j=1}^{n_i} \smsum_{m=1}^{n_k} g_{j1}^{(i)}
      x_{i} h_{1j}^{(i)} h_{m1}^{(k)} y_{k} g_{1m}^{(k)} \\
    &= \smsum_{i=1}^{r} \smsum_{j=1}^{n_i} g_{j1}^{(i)} x_{i} h_{1j}^{(i)}
      h_{j1}^{(i)} y_{i} g_{1j}^{(i)} 
      = \smsum_{i=1}^{r} \smsum_{j=1}^{n_i} g_{j1}^{(i)} g_{11}^{(i)} g_{1j}^{(i)}
      = \smsum_{i=1}^{r} \smsum_{j=1}^{n_i} g_{jj}^{(i)} = 1 \,,
  \end{align*}
  and, similarly, $y x = 1$.  As a result, there exists an inner
  automorphism~$\theta$ of~$T$ given by the rule $\theta(a) = x ay $ for
  all $a \in T$.

  For all $i,j,k$, we compute that 
  \begin{align*}
    xh_{jk}^{(i)} &= \smsum_{s=1}^{r} \smsum_{t=1}^{n_s} g_{t1}^{(s)}
      x_{s} h_{1t}^{(s)} h_{jk}^{(i)} = g_{j1}^{(i)} x_{i}h_{1j}^{(i)} h_{jk}^{(i)} \\
    &= g_{jk}^{(i)} g_{k1}^{(i)} x_{i} h_{1k}^{(i)}
      = \smsum_{s=1}^{n} \smsum_{t=1}^{n_s} g_{jk}^{(i)} g_{t1}^{(s)}
      x_{s} h_{1t}^{(s)} = g_{jk}^{(i)}x \,,
  \end{align*}
  whence $\theta \psi (e_{jk}^{(i)}) = x h_{jk}^{(i)} y = g_{jk}^{(i)}
  = \phi(e_{jk}^{(i)})$.  Since the $e_{jk}^{(i)}$ form a basis for~$S$
  over~$F$, we get that $\theta \psi = \phi$.

  Finally, assume that~$T$ is the limit of a sequence of matricial
  $F$-algebras, and that there is a unit $x \in T$ such that $\theta(a)
  = x a x^{-1}$ for all $a \in T$.  Given any strongly idempotent element
  $e \in T$, we have that $x e = x e x^{-1} x e \in \theta(e) T e$ and
  $e x^{-1} = e x^{-1} x e x^{-1} \in e T \theta(e)$, where $(x e) (e x^{-1})
  = \theta(e)$ and $(e x^{-1}) (x e) = e$, so that $\theta(e) T \cong e T$,
  and therefore $\mathcal{SV}(\theta)(\overline{e T}) =
  \overline{\theta(e) T} = \overline{e T}$.  Using this and
  Remark~\ref{rem:58}, we obtain that $\mathcal{SV}(\theta)$ is the
  identity map on $\mathcal{SV}(S)$.  Thus, $\mathcal{SV}(\phi) =
  \mathcal{SV}(\theta) \mathcal{SV}(\psi) = \mathcal{SV}(\psi)$,
  and we have finished the proof.
\end{proof}

Now we are ready to state the main results of this section.  The
following theorem shows a class of semirings in which the monoid
$\mathcal{SV}(S)$ determines~$S$ up to isomorphism, namely the class
of ultramatricial algebras over a semifield.  Note that its proof is
essentially based on the one of \cite[Thm.~15.26]{g:vnrr}.

\begin{thm}\label{thm:510}
  Let~$S$ and~$T$ be ultramatricial algebras over a semifield~$F$.
  Then $(\mathcal{SV}(S), \overline{S}) \cong (\mathcal{SV}(T),
  \overline{T})$ if and only if $S \cong T$ as $F$-algebras.
\end{thm}

\begin{proof}
  The direction ($\Longleftarrow$) is obvious, so we show the
  direction ($\Longrightarrow$).

  Assume that $f \colon (\mathcal{SV}(S), \overline{S}) \to
  (\mathcal{SV}(T),\overline{T})$ is an isomorphism in $\mathcal{C}$.
  Using Proposition~\ref{prop:57}, we may assume that~$S$ and~$T$ are
  the union of an ascending sequence $S_{1} \subseteq S_{2} \subseteq
  \cdots$ and $T_{1} \subseteq T_{2} \subseteq \cdots$, respectively,
  of matricial subalgebras.  For each $n = 1, 2, \dots$, let 
  $\phi_{n} \colon S_{n} \to S$ and $\psi_{n} \colon T_{n} \to T$ denote
  the corresponding inclusion maps.  We prove first the following
  useful claims.

  \emph{Claim~1}: If $\alpha \colon T_{k} \to S_{n}$ is an $F$-algebra
  homomorphism such that $\mathcal{SV}(\phi_{n} \alpha ) = f^{-1}
  \mathcal{SV}(\psi_{k})$, then there exist an integer $j > k$ and an
  $F$-algebra homomorphism $\beta \colon S_{n}\to T_{j}$ such that
  $\psi_{j} \beta \alpha = \psi_{k}$ and $\mathcal{SV}(\psi_{j} \beta)
  = f \mathcal{SV}(\phi_{n})$.

  \emph{Proof of the claim}.  By Lemma~\ref{lem:59}\,(1), there exists an
  $F$-algebra homomorphism $\beta' \colon S_{n} \to T$ such that
  $\mathcal{SV}(\beta') = f \mathcal{SV}(\phi_{n})$.  Since~$S_{n}$ is a
  free $F$-semimodule of finite rank, we must have $\beta'(S_{n})
  \subseteq T_{i}$ for some~$i$.  Then~$\beta'$ defines an $F$-algebra
  homomorphism $\beta'' \colon S_{n} \to T_{i}$ such that $\psi_{i} \beta''
  = \beta '$, and we have $\mathcal{SV}(\psi_{i} \beta'') 
  = f \mathcal{SV}(\phi_{n})$.  This implies that
  \[ \mathcal{SV}(\psi_{i} \beta'' \alpha ) = f \mathcal{SV} (\phi_{n})
  \mathcal{SV}(\alpha ) = f\mathcal{SV}(\phi_{n}\alpha ) 
  = \mathcal{SV}(\psi_{k}) \,. \]
  Applying Lemma~5.9~(2), there exists an inner automorphism~$\theta$ of~$T$
  such that $\psi_{k} = \theta \psi_{i} \beta''\alpha$.  Since~$T_{i}$ is
  also a free $F$-semimodule of finite rank, there exists an integer 
  $j > k$ such that $\theta(T_{i}) \subseteq T_{j}$.  Then~$\theta$ defines
  an $F$-algebra homomorphism $\theta' \colon T_{i} \to T_{j}$ such
  that $\psi_{j} \theta' = \theta \psi_{i}$.  Set $\beta = \theta'
  \beta''$, so that~$\beta$ is an $F$-algebra homomorphism from~$S_{n}$
  into $S_{j}$ and $\psi_{j} \beta \alpha = \psi_{j} \theta' \beta'' \alpha
  =\theta \psi_{i} \beta'' \alpha = \psi_{k}$.  Using Lemma~5.9~(2), we see
  that \[ \mathcal{SV}(\psi_{j} \beta) = \mathcal{SV}(\psi \theta' \beta'')
  = \mathcal{SV}(\theta \psi_{i} \beta'') = \mathcal{SV}(\psi_{i} \beta'')
  = f\mathcal{SV}(\phi_{n}) \,. \]
  Thus, the claim is proved.

  Similarly, we get the following claim.

  \emph{Claim 2}: If $\alpha \colon S_{n}\to T_{k}$ is an $F$-algebra
  homomorphism such that $\mathcal{SV}(\psi_{k} \alpha) = 
  f \mathcal{SV}(\phi_{n})$, then there exist an integer $m > n$ and
  an $F$-algebra homomorphism $\beta \colon T_{k} \to S_{m}$ such that
  $\phi_{m} \beta \alpha = \phi_{n}$ and $\mathcal{SV}(\phi_{m} \beta )
  = f^{-1} \mathcal{SV}(\psi_{k})$.
  
  We next construct positive integers $n(1) < n(2) < \cdots$ and $F$-algebra
  homomorphisms $\beta_{k} \colon S_{n(k)} \to T$ such that:

  (a) For all $k = 1, 2, \dots$, we have $T_{k} \subseteq \beta_{k}(S_{n(k)})$
  and $\mathcal{SV}(\beta_{k}) = f\mathcal{SV} (\phi_{n(k)})$.

  (b) For all $k = 1, 2, \dots$, it is $\beta_{k}$ injective and
  $\beta_{k+1}$ an extension of $\beta_{k}$.

  By Lemma~\ref{lem:59}\,(1), there exists an $F$-algebra homomorphism
  $\alpha' \colon T_{1} \to S$ such that $\mathcal{SV}(\alpha') = f^{-1}
  \mathcal{SV}(\psi_{1})$.  Because~$T_{1}$ is free $F$-semimodule of finite
  rank, we have $\alpha'(T_{1}) \subseteq S_{n(1)}$ for some positive 
  integer~$n(1)$.  Then~$\alpha '$ defines an $F$-algebra homomorphism 
  $\alpha \colon T_{1} \to S_{n(1)}$ such that $\phi_{n(1)} \alpha = \alpha'$,
  and $\mathcal{SV}(\phi_{n(1)} \alpha ) = f^{-1}\mathcal{SV}(\psi_{1})$. 
  By Claim~1, there exists an integer $j > 1$ and an $F$-algebra
  homomorphism $\beta \colon S_{n(1)} \to S_{j}$ such that $\psi_{j} \beta
  \alpha =\psi_{1}$ and $\mathcal{SV}(\psi_{j} \beta) = f \mathcal{SV}
  (\phi_{n(1)})$.  Then $\beta_{1} := \psi_{j} \beta$ is an $F$-algebra
  homomorphism from~$S_{n(1)}$ to~$T$ such that $\beta_{1} \alpha
  = \psi_{1}$ and $\mathcal{SV}(\beta_{1}) = f\mathcal{SV}(\phi_{n(1)})$.
  Moreover, $T_{1} = \psi_{1}(T_{1}) = \beta_{1} \alpha (T_{1}) \subseteq 
  \beta_{1}(S_{n(1)})$, and we see that~(a) is satisfied for $k = 1$.

  Assume that we have $n(1), \dots, n(k)$ and $\beta_{1}, \dots, \beta_{k}$
  for some positive integer~$k$ such that (a) is satisfied up to~$k$ and
  (b) up to $k - 1$.  Since~$S_{n(k)}$ is a free $F$-semimodule of finite
  rank, there is an integer $i > k$ such that $\beta_{k}(S_{n(k)}) \subseteq
  T_{i}$.  Then~$\beta_{k}$ defines an $F$-algebra homomorphism $\beta '
  \colon S_{n(k)} \to T_{i}$ such that $\psi_{i} \beta' = \beta_{k}$, and
  we note that $\mathcal{SV}(\psi_{i} \beta') = f \mathcal{SV}
  (\phi_{n(k)})$.  By Claim~2, there exist a positive integer 
  $n(k+1) > n(k)$ and an $F$-algebra homomorphism $\delta \colon 
  T_{i} \to S_{n(k+1)}$ such that $\phi_{n(k+1)} \delta \beta'
  = \phi_{n(k)}$ and $\mathcal{SV}(\phi_{n(k+1)} \delta)
  = f^{-1} \mathcal{SV}(\psi_{i})$.  Since $\phi_{n(k+1)} \delta \beta'
  = \phi_{n(k)}$, we get that~$\beta'$ is injective; hence,
  $\beta_{k} = \psi_{i} \beta'$ is also injective.

  Applying Claim~1, there exists an integer $j > i$ and an $F$-algebra
  homomorphism $\gamma \colon S_{n(k+1)} \to S_{j}$ such that $\psi_{j}
  \gamma \delta = \psi_{i}$ and $\mathcal{SV}(\psi_{j} \gamma) = f 
  \mathcal{SV}(\phi_{n(k+1)})$.  Then $\beta_{k+1} := \psi_{j} \gamma$ is
  an $F$-algebra homomorphism from~$S_{n(k+1)}$ into~$T$ such that 
  $\mathcal{SV}(\beta_{k+1}) = f \mathcal{SV}(\phi_{n(k+1)})$. 
  From $\beta_{k+1} \delta = \psi_{j} \gamma \delta = \psi_{i}$ and 
  $i \ge k+1$, we get
  \[ T_{k+1} = \psi_{i}(T_{k+1}) = \beta_{k+1} \delta(T_{k+1}) 
  \subseteq \beta_{k+1} \delta(T_{i}) \subseteq \beta_{k+1} (S_{n(k+1)}) \,. \]
  Finally, since $\phi_{n(k+1)} \delta \beta' = \phi_{n(k)}$ and
  $\beta_{k+1} \delta \beta' = \psi_{i} \beta' = \beta_{k}$, we obtain
  that~$\beta_{k+1}$ is an extension of~$\beta_{k}$.  Therefore, (a) holds
  for $k+1$ and (b) for~$k$, so that the induction works.

  Since $n(1) < n(2) < \cdots$ and $n(k) \ge k$, we immediately get
  $\bigcup S_{n(k)} = S$.  As a result, the $\beta_{k}$ induce an
  injective $F$-algebra homomorphism $\beta \colon S \to T$ such that
  $\beta \phi_{n(k)} = \beta_{k}$.  Then $T_{k} \subseteq
  \beta_{k}(S_{n(k)}) = \beta(S_{n(k)}) \subseteq \beta(S)$ for all~$k$,
  hence, $T = \beta(S)$.  Thus~$\beta$ is surjective and,
  therefore, establishes an isomorphism $S \cong T$.
\end{proof}

Finally, we deduce that ultramatricial algebras over a semifield are
characterized also by their monoid of finitely generated projective
semimodules.

\begin{thm}\label{thm:511}
  Let~$S$ and~$T$ be ultramatricial algebras over a semifield~$F$.
  Then, there exists a monoid isomorphism $f \colon \mathcal{V}(S)
  \to \mathcal{V}(T)$ such that $f(\overline{S}) = \overline{T}$
  if and only if $S \cong T$ as $F$-algebras.
\end{thm}

\begin{proof}
  Again, the direction ($\Longleftarrow$) is obvious, and we show the
  direction ($\Longrightarrow$).

  Let~$P$ be a finitely generated strongly projective right
  $S$-semimodule.  Then there is a positive integer~$n$ such that
  $S^{n} \cong P \oplus Q$ for some finitely generated projective
  right $S$-semimodule, \ie, we have $\overline{n S} = \overline{P}
  + \overline{Q}$ in $\mathcal{V}(S)$, whence
  \[ \overline{nT} = f(\overline{nS}) = f(\overline{P}) + f(\overline{Q}) \]
  in $\mathcal{V}(T)$.  This implies that $f(\overline{P}) \in
  \mathcal{VS}(T)$, so~$f$ induces an isomorphism from 
  $(\mathcal{SV}(S), \overline{S})$ onto $(\mathcal{SV}(T), 
  \overline{T})$ in the category~$\mathcal{C}$.  Now, applying
  Theorem~\ref{thm:510}, we immediately get that $S \cong T$ as
  $F$-algebras, as desired.
\end{proof}

\section{The $K_{0}$-group characterization 
  of ultramatricial algebras\\ over semifields}\label{sec:6}

The main goal of this section is to investigate Grothendieck's
$K_{0}$-groups on finitely generated projective semimodules, whose
study was initiated by Di Nola and Russo~\cite{dr:tstatmaas},
and to introduce and examine $K_{0}$-theory on finitely generated
strongly projective semimodules, as well as to establish ``semiring''
analogs of Elliott's celebrated classification theorem for
ultramatricial algebras over an arbitrary field \cite%
{elliott:otcoilososfa}.  Consequently, we classify zerosumfree
congruence-semisimple semirings in terms of $K_{0}$-theory.

We begin this section by recalling the $K_{0}$-group of a semiring
which was mentioned by Di Nola and Russo in \cite[Sec.~4]{dr:tstatmaas}.

\begin{defn}[{cf.~\cite[Sec.~4]{dr:tstatmaas}}]
  Let~$S$ be a semiring.  The \emph{Grothendieck group} $K_{0}(S)$ is
  the additive abelian group presented by the set of generators
  $\mathcal{V}(S)$ and the following set of relations:
  $\overline{P \oplus Q} = \overline{P} + \overline{Q}$ for all 
  $\overline{P}, \overline{Q} \in \mathcal{V}(S)$.
\end{defn}


\begin{rem}\label{rem:62}
  The Grothendieck group $K_{0}(S)$ can be described as follows.
  \begin{enumerate}[label=(\arabic*)]
  \item Let~$G$ be the free abelian group generated by $\overline{P}
    \in \mathcal{V}(S)$, and~$H$ the subgroup of~$G$ generated by
    \[ \overline{P \oplus Q} - \overline{P} - \overline{Q} \,, \]
    where $\overline{P}, \overline{Q}\in \mathcal{V}(S)$.
    Then $K_{0}(S) = G / H$, and we denote by $[P]$ the image of
    $\overline{P}$ in $K_{0}(S)$.
  \item A general element of $K_{0}(S)$ has the form \[ x \,=\, [P_{1}]
    + \ldots + [P_{m}] - [Q_{1}] - \ldots - [Q_{n}] \,=\, [P] - [Q] \,, \]
    with $P := P_{1} \oplus \cdots \oplus P_{m}$, $Q := Q_{1} \oplus
    \cdots \oplus Q_{n}$ and $\overline{P}, \overline{Q} \in \mathcal{V}(S)$.
  \item Defining on $X := \mathcal{V}(S) \!\times\! \mathcal{V}(S)$ an
    equivalence relation by
    \[ (\overline{P}, \overline{Q}) \sim (\overline{P'}, \overline{Q'})
    \quad \Longleftrightarrow \quad \exists \overline{T} \in \mathcal{V}(S)
    : \, P \oplus Q' \oplus T \cong P' \oplus Q \oplus T \,, \]
    it is a routine matter to check that $X / \!\!\sim\,\, :=
    \{ [ \overline{P}, \overline{Q} ] \mid \overline{P}, \overline{Q}
    \in \mathcal{V}(S) \}$ becomes an abelian group by defining
    $[ \overline{P}, \overline{Q} ] + [ \overline{T}, \overline{U} ] :=
    [ \overline{P} \oplus \overline{T} , \overline{Q} \oplus \overline{U} ]$
    for $\overline{P}, \overline{Q}, \overline{T}, \overline{U} \in
    \mathcal{V}(S)$, and that there is a group isomorphism given by
    \[ K_0(S) \to  X / \!\!\sim \,, \quad 
    [P] - [Q] \mapsto [ \overline{P}, \overline{Q} ] \,. \]
  \item For any finitely generated projective right $S$-semimodules~$P$
    and~$Q$ we have $[P] = [Q] \in K_{0}(S)$ if and only if $P \oplus T
    \cong Q \oplus T$ for some finitely generated projective right
    $S$-semimodule~$T$.

    This follows immediately from~(3); see also \cite[Prop.~I.6.1]%
    {lam:spopm} for a direct proof given in the case of rings, which
    serves in our semiring setting as well.
  \end{enumerate}
\end{rem}

Notice that if $\phi \colon R \to S$ is a semiring homomorphism, 
then~$\phi$ induces a well-defined group homomorphism $K_{0}(\phi)
\colon K_{0}(R) \to K_{0}(S)$ such that $K_{0}(\phi)([P]) = [\phi_{\#}(P)]
= [P \otimes_{R} S]$.  From this observation, we easily check that~$K_{0}$
gives a covariant functor from the category of semirings to the category
of abelian groups.  This fact was also mentioned by Di Nola and
Russo~\cite{dr:tstatmaas}.

Furthermore, we note the following useful fact.  The proof is quite
similar as it was done in the one of Proposition~\ref{prop:54} (also,
we can refer to \cite[Prop.~15.11, Prop.~15.13]{g:vnrr}); hence, we
will not reproduce it here.

\begin{prop}\label{prop:63}
  The functor $K_{0}(-) \colon \mathcal{SR} \to \mathcal{A}$ preserves
  direct limits and finite products, where $\mathcal{A}$ is the
  category of abelian groups.
\end{prop}

We next consider the $K_{0}$-group of finitely generated strongly
projective semimodules over a semiring. Similarly to the group
$K_{0}(S)$ of a semiring~$S$, we introduce the following notion.

\begin{defn}
  Let~$S$ be a semiring.  The \emph{Grothendieck group} $SK_{0}(S)$ is
  the additive abelian group presented by the set of generators
  $\mathcal{SV}(S)$ and the following set of relations:
  $\overline{P \oplus Q} = \overline{P} + \overline{Q}$ for all
  $\overline{P}, \overline{Q} \in \mathcal{SV}(S)$.
\end{defn}

The Grothendieck group $SK_0(S)$ can also be described as follows.
Let~$G'$ be the free abelian group generated by $\overline{P} \in
\mathcal{SV}(S)$, and $H'$ the subgroup of~$G'$ generated by
\[ \overline{P \oplus Q} - \overline{P} - \overline{Q} \,, \]
where $\overline{P}, \overline{Q} \in \mathcal{SV}(S)$.  Then 
$SK_0(S) = G' / H'$, and we denote by~$\widehat{P}$ the image 
of~$\overline{P}$ in~$SK_0(S)$.

Similarly to the case of the group $K_{0}(S)$, a general element of
$SK_{0}(S)$ may be written in the form
\[ x = \widehat{P} - \widehat{Q} \,, \]
where~$P$ and~$Q$ are finitely generated strongly projective right
$S$-semimodules.  We may also choose a finitely generated strongly
projective right $S$-semimodule~$Q'$ such that $Q \oplus Q'
\cong S^{n}$ for some~$n$, and rewrite
\[ x = \widehat{P \oplus Q'} - \widehat{Q \oplus Q'}
= \widehat{P'} - \widehat{S^{n}} \,, \]
where $P' = P \oplus Q'$.

Again, it is not hard to see that $SK_{0}(-)$ defines a covariant
functor from the category of semirings to the category of abelian
groups.  Furthermore, we have the following lemma, whose proof is done
similarly to the ones of Remark~\ref{rem:62}, and hence, we will not
reproduce it here.

\begin{lem}[{cf.~\cite[Prop.~I.6.1]{lam:spopm}}]\label{lem:65}
  Let~$P$ and~$Q$ be finitely generated strongly projective right
  $S$-semimodules.  Then, the following are equivalent:
  \begin{enumerate}[label=(\arabic*)]
  \item $\widehat{P} = \widehat{Q} \in SK_{0}(S)$;
  \item $P \oplus T \cong Q \oplus T$ for some finitely generated
    strongly projective right $S$-semimodule~$T$;
  \item there exists a positive integer~$n$ such that 
  $P \oplus S^{n} \cong Q \oplus S^{n}$.
  \end{enumerate}
\end{lem}

\begin{rem}\label{rem:66}
  Let us note the following simple facts.
  \begin{enumerate}[label=(\arabic*)]
  \item For any semiring~$S$, there is always the canonical group
    homomorphism $\jmath \colon SK_{0}(S) \to K_{0}(S)$, defined by
    $\jmath(\widehat{P}) = [P]$.
  \item Let~$S$ be a semiring all of whose finitely generated
    projective right modules are free.  Then the groups $SK_{0}(S)$ and
    $K_{0}(S)$ are cyclic groups generated by $\widehat{S}$ and $[S]$,
    respectively.  And, if in addition~$S$ has the IBN property, then
    those groups are exactly $\mathbb{Z}$.
  \item The group $SK_{0}(\mathbb{B})$ is isomorphic to the free
    abelian group $\mathbb{Z}$, but $K_{0}(\mathbb{B})$ contains as a
    subgroup a free abelian group with countably infinite basis.
    Indeed, the first fact follows from Example~\ref{exa:53}\,(1), and
    the latter follows from the following proposition.
  \end{enumerate}
\end{rem}

\begin{prop}\label{prop:67}
  For an additively idempotent commutative semiring~$S$, the group
  $K_{0}(S)$ contains a free abelian group with countable basis as a
  subgroup.
\end{prop}

\begin{proof}
  We first prove the statement for the case when $S = \mathbb{B}$.
  For any prime number~$p$, consider the subset $Q_{p} := \{ 0, 1,
  \dots, p\!-\!1 \}$ of $\mathbb{Z}^{+}$.  From Fact~\ref{fac:46} we see that
  the monoid $(Q_{p}, \max)$ is a projective $\mathbb{B}$-semimodule.
  We denote by~$G$ the subgroup of $K_{0}(S)$ generated by all
  elements~$[Q_{p}]$ and  show that the countably infinite set
  $\{ [Q_{p}] \mid p \text{ is prime} \}$ is a basis of~$G$.

  Indeed, assume that $\sum_{i=1}^{r} n_{i} [Q_{p_{i}}] = 0$ in
  $K_{0}(\mathbb{B})$, where the $n_{i}$ are integers and the $p_{i}$
  are pairwise distinct prime numbers.  We may assume that
  $n_{1}, \dots, n_{k} \ge 0$ and $n_{k+1}, \dots, n_{r} \le 0$, so
  that, writing $m_{j} = -n_{j}$ for $k \!+\! 1 \le j \le r$, we have
  \[ \smsum_{i=1}^{k} n_{i} [Q_{p_{i}}] = \smsum_{j=k+1}^{r} m_{j} [Q_{p_{j}}] \,, \]
  and hence, $\bigoplus_{i=1}^{k} Q_{p_{i}}^{n_{i}} \oplus T
  = \bigoplus_{j=k+1}^{r} Q_{p_{j}}^{m_{j}} \oplus T$ for some finitely
  generated projective $\mathbb{B}$-semimodule~$T$.

  Since every finitely generated projective $\mathbb{B}$-semimodule is
  finite, we get that $|\bigoplus_{i=1}^{k} Q_{p_{i}}^{n_{i}} \oplus T|
  = |\bigoplus_{j=k+1}^{r} Q_{p_{j}}^{m_{j}} \oplus T|$, and hence, 
  $|\bigoplus_{i=1}^{k} Q_{p_{i}}^{n_{i}}| = |\bigoplus_{j=k+1}^{r} 
  Q_{p_{j}}^{m_{j}}|$.  This implies that $p_{1}^{n_{1}} \ldots p_{k}^{n_{k}}
  = p_{k+1}^{m_{k+1}} \ldots p_{r}^{m_{r}}$, and hence, we must have 
  $n_{i} = 0$ for all $i = 1, \dots, r$.  Therefore, $K_{0}(\mathbb{B})$
  contains the subgroup~$G$ which is isomorphic to the free abelian group
  with countable basis $\{ p \in \mathbb{Z}^{+} \mid p \text{ is prime} \}$.
  
  Consider now the case when~$S$ is an arbitrary additively idempotent
  commutative semiring.  One may readily find a maximal congruence~$\rho$
  on~$S$ by using Zorn's lemma.  We then have that the additively
  idempotent commutative semiring $T : = S / \rho$ has only the
  trivial congruences.  By \cite[Thm.~10.1]{bshhurtjankepka:scs}, $T$~is
  the Boolean semifield $\mathbb{B}$.  Let $\imath \colon \mathbb{B} \to S$
  and $\pi \colon S \to \mathbb{B}$ be the canonical injection and
  surjection, respectively.  Since $\pi \circ \imath = \on{id}_{\mathbb{B}}$,
  we must have that $K_{0}(\pi) K_{0}(\imath) = \on{id}_{K_{0}(\mathbb{B})}$, by
  the functorial property of $K_{0}$.  This implies that $K_{0}(\imath)
  \colon K_{0}(\mathbb{B}) \to K_{0}(S)$ is an injective group homomorphism,
  and hence, we may consider $K_{0}(\mathbb{B})$ as a subgroup 
  of~$K_{0}(S)$ and finish the proof.
\end{proof}

The homomorphism $\jmath $ of Remark~\ref{rem:66}\,(1) is, in general,
not an isomorphism in semiring setting.

\begin{lem}\label{lem:68}
  The canonical homomorphism $\jmath \colon SK_{0}(\mathbb{B}) \to
  K_{0}(\mathbb{B})$ is injective but not surjective.
\end{lem}

\begin{proof} 
  Assume that $x = \widehat{P} - \widehat{Q} \in SK_{0}(\mathbb{B})$
  such that $\jmath(x) = 0$.  We then have that $[P] = [Q] \in
  K_{0}(\mathbb{B})$, so $P \oplus T \cong Q \oplus T$ for some
  finitely generated projective $\mathbb{B}$-semimodule~$T$, by 
  Remark~\ref{rem:62}\,(4).  For~$P$ and~$Q$ are finitely generated 
  strongly projective $\mathbb{B}$-semimodule and Theorem~\ref{thm:45},
  there exist nonnegative integers~$m$ and~$n$ such that $P \cong
  \mathbb{B}^{m}$ and $Q \cong \mathbb{B}^{n}$.  Furthermore,
  since~$T$ is a finitely generated $\mathbb{B}$-semimodule, we
  immediately get that~$T$ is a finite set.  Then, from the equality
  $P \oplus T \cong Q \oplus T$, we must have that $|\mathbb{B}^{m}
  \oplus T| = |\mathbb{B}^{n} \oplus T|$; hence, $m = n$.  This
  implies that $P \cong Q$, that means, $x = 0 \in
  SK_{0}(\mathbb{B})$.  Therefore, $\jmath$ is injective.

  From Proposition~\ref{prop:67} and since $SK_0(\mathbb{B}) \cong
  \mathbb{Z}$ by Remark~\ref{rem:66}\,(2), we see that~$\jmath$ is not
  surjective.  For the reader's convenience, we also give a direct
  argument.

  Assume that~$\jmath$ is surjective, and consider the projective
  $\mathbb{B}$-semimodule $P = \{ (0,0), (0,1), (1,1) \}$ as in 
  Example~\ref{exa:47}.  We then have that there exists an element
  $x = \widehat{A} - \widehat{B} \in SK_{0}(\mathbb{B})$ such that
  $\jmath(x) = [P]$, that means,
  \[ A \oplus C \cong B \oplus P \oplus C \]
  for some finitely generated projective $\mathbb{B}$-semimodule~$C$.
  Since~$A$, $B$ and~$C$ are finite, we get that $|A| = |B \oplus P|$.
  By Theorem~\ref{thm:45}, we write~$A$ and~$B$ of the form $A \cong
  \mathbb{B}^{m}$ and $B \cong \mathbb{B}^{n}$ for some
  nonnegative integers~$m$ and~$n$.  From these observations, we must
  have that $|P| = 2^{m-n}$, a contradiction.  Thus~$\jmath$ is not
  surjective, as claimed.
\end{proof}

In the next theorem we provide a criterion for checking the isomorphism
property of the homomorphism~$\jmath$ in a division semiring setting.
Before doing this, we need some useful notions and facts.  Following
\cite{agg:liotsab}, a family~$X$ of elements in a semimodule~$M$ over
a semiring~$S$ is \emph{weakly linearly independent} if there is no
element in~$X$ that can be expressed as a linear combination of other
elements of~$X$.  We define the \emph{weak dimension} of~$M$, denoted
by $\dim_{w}(M)$, as the minimum cardinality of a weakly linearly
independent generating family of~$M$.  It is not hard to see that the
weak dimension of a semimodule~$M$ is equal to the minimum cardinality
of a minimal generating family, or the minimum cardinality of any
generating family of~$M$.  Recall that a semiring~$S$ is \emph{entire}
if $a b = 0$ implies that $a = 0$ or $b = 0$ for any $a, b \in S$.

\begin{lem}\label{lem:69}
  Let~$S$ be a zerosumfree entire semiring, and~$P, Q$ finitely
  generated projective right $S$-semimodules.  Then $\dim_{w}(P \oplus Q)
  = \dim_{w}(P) + \dim_{w}(Q)$.
\end{lem}

\begin{proof}
  Notice that $\dim_{w}(A \oplus B) \le \dim_{w}(A) + \dim_{w}(B)$ holds
  for any finitely generated right $S$-semimodules~$A$ and~$B$, so it
  suffices to prove the converse inequality for finitely generated
  projective right $S$-semimodules, which we may assume to be nonzero.
  It is easy to see that every projective right $S$-semimodule~$T$ is
  zerosumfree, and satisfies that $x s = 0$ implies $x = 0$ or
  $s = 0$, for $x \in T$, $s \in S$.

  Now let $d := \dim_{w}(P \oplus Q)$ and let $X = \{( x_{1} ,y_{1}),
  \dots, (x_{d}, y_{d}) \}$ be a minimal generating family of $P \oplus Q$
  for some $x_{i} \in P$, $y_{i} \in Q$.  We may assume that
  \[ X = \big\{ (x_1, 0), \dots, (x_k, 0), (0, y_{k+1}),
  \dots, (0, y_n), (x_{n+1}, y_{n+1}), \dots, (x_d, y_d) \big\} \]
  with all $x_i, y_j \ne 0$, for some $0 \le k \le n \le d$.

  For each $n \!+\!1 \le j \le d$, we have $(x_{j}, 0) \in P \oplus Q$
  and hence we can write $(x_{j}, 0) = \sum_{i=1}^d (x_{i}, y_{i}) s_{i}$
  for some $s_{i} \in S$,  so that \[ x_j = \smsum_{i=1}^{d} x_{i} s_{i} \in P
  \quad \text{ and } \quad 0 = \smsum_{i=1}^d y_{i} s_{i} \in Q \,. \]
  As~$Q$ is zerosumfree we have $y_i s_i = 0$ for all~$i$, and since
  $y_{k+1}, \dots, y_{d} \ne 0$ we get that $s_{k+1} = \ldots = s_{d} = 0$;
  thus we obtain $(x_{j}, 0) = \sum_{i=1}^{k} (x_{i}, 0) s_{i}$. 
  Similarly, we have that $(0, y_{j}) = \sum_{j=k+1}^{n} (0, y_{j}) r_{j}$
  for some $r_{j} \in S$, whence \[ (x_{j}, y_{j}) = \smsum_{i=1}^{k}(x_{i}, 0)
  s_{i} + \smsum_{j=k+1}^{n} (0, y_{j}) r_{j} \,. \]
  From this observation and the minimality of~$X$, we infer that
  $n = d$.

  It is then easy to see that $\{ x_{1}, \dots, x_{k} \}$ and $\{ y_{k+1},
  \dots, y_{d} \}$ are generating families of~$P$ and~$Q$, respectively.
  This implies that $k \ge \dim_{w}(P)$ and $d - k \ge \dim_{w}(Q)$,
  whence $\dim_{w}(P \oplus Q) = d \ge \dim_{w}(P) + \dim_{w}(Q)$,
  as desired.
\end{proof}

\begin{thm}\label{thm:610}
  For an arbitrary division semiring~$D$ the following is true:
  \begin{enumerate}[label=(\arabic*)]
  \item The canonical homomorphism $\jmath \colon SK_{0}(D) \to K_{0}(D)$
    is injective.
  \item The homomorphism $\jmath \colon SK_{0}(D) \to K_{0}(D)$ is an
    isomorphism if and only if~$D$ is a weakly cancellative division
    semiring.
  \end{enumerate}
\end{thm}

\begin{proof}
  Note first that~$D$ is either a division ring or a zerosumfree
  division semiring.  Also, if~$D$ is a division ring, then the group
  homomorphism $\jmath \colon SK_{0}(D) \to K_{0}(D)$ is always an
  isomorphism, since every right $D$-semimodule is free; that means,
  the statements are obvious.  Consider now the case when~$D$ is a
  zerosumfree division semiring.  We then have a semiring homomorphism
  $\pi \colon D\to \mathbb{B}$, defined by $\pi(0) = 0$ and $\pi(x) = 1$
  for all $0 \ne x \in D$, and the following diagram
  \[ \xymatrix{ SK_{0}(D) \ar[r]^{\jmath} \ar[d]_{SK_{0}(\pi)} 
    & K_{0}(D) \ar[d]^{K_{0}(\pi)} \\
    SK_{0}(\mathbb{B}) \ar[r]_{\jmath} & K_{0}(\mathbb{B}) } \]
  is commutative.  By Proposition~\ref{prop:55}, we obtain that $SK_{0}(D)$
  is the free abelian group with basis $\widehat{D}$.

  (1) Assume that $\jmath (n \widehat{D}) = [0] \in K_{0}(D)$ for some
  nonnegative integer~$n$.  Since the diagram above is commutative and
  $\jmath \colon SK_{0}(\mathbb{B}) \to K_{0}(\mathbb{B})$ is injective
  by Lemma~\ref{lem:68}, we get that $SK_{0}(\pi)(n \widehat{D})
  = \widehat{0} \in SK_{0}(\mathbb{B})$, that means,
  \[ \widehat{\mathbb{B}^{n}} = \widehat{\pi_{\#}(D^{n})} = 
  \widehat{0} \in SK_{0}(\mathbb{B}) \,, \]
  so that $\mathbb{B}^{n} \oplus \mathbb{B}^{m} \cong \mathbb{B}^{m}$ for
  some nonnegative integer~$m$, by Lemma~\ref{lem:65}.  This implies that
  $n = 0$, whence $\jmath \colon SK_{0}(D) \to K_{0}(D)$ is injective.

  (2) The sufficient condition follows immediately from
  \cite[Thm.~3.2]{ik:ospopsops} which shows that every projective
  right semimodule over a weakly cancellative division semiring is
  always free. In order to prove the necessary condition, we assume
  that the group homomorphism $\jmath \colon SK_{0}(D)\to K_{0}(D)$ is
  an isomorphism, and~$D$ is not a weakly cancellative division
  semiring.  Then there exists an element $s \in D$ such that
  $s \ne 1$ and $1 + 1 = 1 + s$.  Suppose that $s = 0$.  Then~$D$
  contains~$\mathbb{B}$ as a subsemiring.  Letting $\imath \colon
  \mathbb{B} \to D$ be the canonical injection, we have that
  $\pi \circ \imath = \on{id}_{\mathbb{B}}$, whence $K_{0}(\pi)
  K_{0}(\imath) = \on{id}_{K_{0}(\mathbb{B})}$.  This implies
  that~$K_{0}(\pi)$ is surjective, so that $\jmath \colon
  SK_{0}(\mathbb{B} ) \to K_{0}(\mathbb{B})$ is also surjective,
  contradicting Lemma~\ref{lem:68}.  Therefore, we must have that
  $s \ne 0$.

  Consider the subsemimodule $P=\{(d,d)a+(ds,d)b\mid a,b\in D\}$ of
  the free right $D$-semimodule $D^{2}$, where $d=(1+1)^{-1}$. We then
  have that $P \in |\mathcal{M}_{D}|$ is finitely generated projective,
  since $d + d = 1 = d + d s$ and thus the matrix $\big(
  \begin{smallmatrix} d & ds \\ d & d \end{smallmatrix} \big)$ is an
  idempotent one.  Since $\jmath \colon SK_{0}(D) \to K_{0}(D)$ is
  an isomorphism, there exists a nonnegative integer~$n$ such that
  $\jmath (n \widehat{D}) = [P] \in K_{0}(D)$, and hence,
  $D^{n} \oplus Q \cong P \oplus Q$ for some finitely generated
  projective right $D$-semimodule~$Q$.  This implies that $K_{0}(\pi)
  ([D^{n} \oplus Q]) = K_{0}(\pi) ([P\oplus Q])$, that is, 
  \[ \pi_{\#}(D^{n}) \oplus \pi_{\#}(Q) \oplus T \,\cong\,
  \pi_{\#}(P) \oplus \pi_{\#}(Q) \oplus T \]
  for some finitely generated projective right $\mathbb{B}$-semimodule~$T$.
  Furthermore, we have that $\pi_{\#}(D^{n}) \cong \mathbb{B}^{n}$ and
  $\pi_{\#}(P) \cong \mathbb{B}$, so that
  \[ \mathbb{B}^{n} \oplus \pi_{\#}(Q) \oplus T \,\cong\, \mathbb{B}
  \oplus \pi_{\#}(Q) \oplus T \,. \] 
  Consequently, we must have that $n = 1$; hence, $D \oplus Q \cong P
  \oplus Q$.  Now, applying Lemma~\ref{lem:69}, we obtain that
  $\dim_{w}(P) = 1$, so that~$P$ is free.

  On the other hand, by \cite[Prop.~3.1]{ik:ospopsops}, $P$ is not a
  free semimodule, and with this contradiction we end the proof.
\end{proof}

Theorem~\ref{thm:610} provokes a quite natural and, in our view,
interesting question.

\begin{prob}
  Describe the class of all semirings~$S$ having the homomorphism
  $\jmath \colon SK_{0}(S) \to K_{0}(S)$ to be an isomorphism.  Is this
  class axiomatizable in the first order semiring language?
\end{prob}

In order to discuss the combined structure on $SK_0(S)$ and $K_0(S)$,
we require the following definitions. Recall \cite[p.~203]{g:vnrr}
that a \emph{cone} in an abelian group~$G$ is an additively closed
subset~$C$ such that $0 \in C$.  Any cone~$C$ in~$G$ determines a
pre-order~$\le$ (\ie, a reflexive, transitive relation) on~$G$, which
is translation-invariant (\ie, $x \le y$ implies $x + z \le y + z$),
by letting $x \le y$ if and only if $y - x \in C$.  Conversely, any
translation-invariant pre-order~$\le$ on~$G$ arises in this fashion,
from the cone $\{ x \in G \mid x \ge 0 \}$.

A \emph{pre-ordered abelian group} is a pair $(G, \le)$, where~$G$ is
an abelian group and~$\le$ is a translation-invariant pre-order
on~$G$.  When there is no danger of confusion as to the pre-order
being used, we refer to~$G$ itself as a pre-ordered abelian group,
and we write~$G^+$ for the cone $\{ x \in G \mid x \ge 0 \}$.

An \emph{order-unit} in a pre-ordered abelian group~$G$ is an element
$u \in G^+$ such that for any $x \in G$, there exists a positive
integer~$n$ with $x \le n u$.  We note that if~$G$ has an order-unit,
then any element $x \in G$ can be written as the difference of two
elements of~$G^+$, whence~$G$ is generated as a group by~$G^+$.

Throughout this section, we use~$\mathcal{P}$ to denote the following
category.  The objects of~$\mathcal{P}$ are pairs $(G, u)$ such
that~$G$ is a pre-ordered abelian group and~$u$ is an order-unit
in~$G$.  The morphisms in~$\mathcal{P}$ from an object $(G, u)$ to an
object $(H, v)$ are the monotone (\ie, order-preserving) group
homomorphisms $f \colon G \to H$ such that $f(u) = v$.
Unspecificed categorical terms applied to pre-ordered abelian groups
are to be interpreted in~$\mathcal{P}$.  For example, $(G, u)
\cong (H, v)$ means that there is a group isomorphism
$f \colon G \to H$ with $f(u) = v$ and $f, f^{-1}$ monotone
(equivalently, there exists a group isomorphism $f \colon G \to H$
such that $f(u) = v$ and $f(G^{+}) = H^{+}$).

\begin{defn}[{cf.~\cite[Def., p.~203]{g:vnrr}}]
  Given a semiring~$S$, there is a natural way to make $SK_{0}(S)$
  into a pre-ordered abelian group with order-unit, as follows.
  First, let~$X$ denote the class of all finitely generated strongly
  projective right $S$-semimodules, and define $SK_{0}(S)^{+}
  = \{ \widehat{P} \mid P \in X \}$.  It is clear that
  $SK_{0}(S)^{+}$ is a cone in $SK_{0}(S)$, and we refer to the
  pre-order on $SK_{0}(S)$ determined by this cone as the natural
  pre-order on $SK_{0}(S)$.  Explicitly, we have
  $\widehat{A} - \widehat{B} \le \widehat{C} - \widehat{D}$ in
  $SK_{0}(S)$ if and only if $A \oplus D \oplus E \oplus S^{n} 
  \cong C \oplus B \oplus S^{n}$ (as right $S$-semimodules) for some
  $E \in X$ and some positive integer~$n$.
\end{defn}

For any semiring~$S$, observe that~$\widehat{S}$ is an order-unit in
the pre-ordered abelian group $SK_{0}(S)$.  Therefore, we have an
object $(SK_{0}(S), \widehat{S})$ in the category~$\mathcal{P}$
defined above.  Given any semiring homomorphism $\phi \colon R \to S$,
note that $SK_{0}(\phi)$ maps $SK_{0}(R)^{+}$ into $SK_{0}(S)^{+}$, and
that $SK_{0}(\widehat{R}) = \widehat{S}$, so that $SK_{0}(\phi)$ is a
morphism in $\mathcal{P}$ from $(SK_{0}(R), \widehat{R})$ to
$(SK_{0}(S), \widehat{S})$.  Thus, $(SK_{0}(-), \widehat{-})$ defines
a covariant functor from the category of semirings to the
category~$\mathcal{P}$.

We present the following examples in order to illustrate these notions.

\begin{exas}\label{exa:612}
  Let~$F$ be any semifield.  Then we have:
  \begin{enumerate}[label=(\arabic*)]
  \item $(SK_{0}(F), \widehat{F}) \cong (\mathbb{Z}, 1)$;
  \item $(SK_{0}(M_{n}(F)), \widehat{M_{n}(F)}) \cong 
    (\mathbb{Z}, n) \cong (\frac 1n \mathbb{Z}, 1)$, for any 
    positive integer~$n$.
  \end{enumerate}
\end{exas}

\begin{proof}
  (1) By Theorem~\ref{thm:45}, each element of $SK_0(F)$ is uniquely
  written in the form $n \widehat{F}$ with $n\in \mathbb{Z}$, giving
  an isomorphism $(SK_{0}(F), \widehat{F}) \cong (\mathbb{Z}, 1)$, $n
  \widehat{F} \mapsto n$.

  (2) Set $S := M_n(F)$.  As was shown in \cite[Theorem~5.14]{kat:thcos},
  the functors $G : \mathcal{M}_{S} \leftrightarrows \mathcal{M}_{F} : H$
  given by $G(A) = A e_{11}$ and $H(B) = B^{n}$ establish an equivalence
  of the semimodule categories~$\mathcal{M}_{S}$ and~$\mathcal{M}_{F}$,
  where~$e_{11}$ is the matrix unit in $M_{n}(F)$.  Then the restriction
  $\alpha|_{SK_0(S)} \colon SK_0(S) \to SK_0(F)$ is a group isomorphism
  with $\alpha(\widehat{S}) = \widehat{F^n} = n \widehat{F}$ in~$SK_0(F)$.
  This shows that $(SK_{0}(S), \widehat{S}) \cong (SK_0(F), n \widehat{F})$,
  and hence, we readily get that $(SK_{0}(S), \widehat{S}) \cong 
  (\mathbb{Z}, n)$, using~(1).
  The group isomorphim $\phi \colon \mathbb{Z} \to \frac{1}{n}%
  \mathbb{Z}$, defined by $\phi(1) = \frac{1}{n}$, induces an
  isomorphism $(\mathbb{Z},n) \cong (\frac{1}{n}\mathbb{Z},1)$.
\end{proof}

For any a semiring~$S$, similarly to the case of $SK_0(S)$, we may
make $K_0(S)$ into a pre-ordered abelian group by defining the cone
$K_0(S)^+$ to be the set of all~$[P]$, where~$P$ is a finitely
generated projective right $S$-semimodule.  Explicitly, we have
$[A] - [B] \le [C] - [D]$ in $K_{0}(S)$ if and only if
$A \oplus D \oplus E \oplus F \cong C \oplus B \oplus F$ (as right
$S$-semimodules) for some finitely generated projective right
$S$-semimodules~$E$ and~$F$.  However, $K_{0}(S)$ has no order-units
in general.

\begin{prop}\label{prop:613}
  For any additively idempotent commutative semiring~$S$, the
  pre-ordered group $K_{0}(S)$ has no order-units.
\end{prop}

\begin{proof}
  We first prove the statement for the case when $S = \mathbb{B}$.
  Assume that~$[A]$ is an order-unit in $K_{0}(\mathbb{B})$.  For any
  prime number~$p$ consider the subset $Q_{p} := \{ 0, 1, \dots, 
  p \!-\! 1 \}$ of $\mathbb{Z}^{+}$ and the monoid $(Q_{p}, \max)$,
  which is a projective $\mathbb{B}$-semimodule by Fact~\ref{fac:46}.  
  Since~$[A]$ is an order-unit, we have that $[Q_{p}] \le m [A]$ in
  $K_{0}(\mathbb{B})$ for some positive integer~$m$, which means that
  $Q_{p} \oplus B \oplus C \cong A^{m} \oplus C$ for some finitely
  generated, thus finite, projective $\mathbb{B}$-semimodules~$B$ and~$C$.
  We deduce that $|Q_{p} \oplus B \oplus C| = |A^{m} \oplus C|$,
  so that $|A|^{m} = |Q_{p} \oplus B| = |Q_{p}| |B|$.  This implies that
  $p = |Q_{p}|$ divides $|A|$ for all primes~$p$, which is a contradiction.

  Now if~$S$ is an arbitrary additively idempotent commutative semiring,
  by the same reasoning as in the proof of Proposition~\ref{prop:67}, we
  have homomorphisms $\imath \colon \mathbb{B} \to S$ and $\pi \colon S
  \to \mathbb{B}$ such that $K_0(\pi) K_0(\imath) = \on{id}_{K_0(\mathbb{B})}$.
  Suppose~$U$ to be an order-unit in $K_{0}(S)$.  Then, for each
  $X \in K_{0}(\mathbb{B})$, there exists a positive integer~$n$ such
  that $K_{0}(\imath)(X) \le n U$ in $K_{0}(S)$, whence
  \[ X = K_{0}(\pi) K_{0}(\imath)(X) \le n K_{0}(\pi)(U) \]
  in $K_{0}(\mathbb{B})$, which means that $K_{0}(\pi)(U)$ is an
  order-unit in $K_{0}(\mathbb{B})$, contradicting the fact above,
  finishing our proof.
\end{proof}

Izhakian, Knebusch and Rowen~\cite{ikr:domlzs} develop a theory of the
decomposition socle for zerosumfree semimodules; and consequently, the
authors prove that a zerosumfree semiring~$S$ is a finite direct sum
of indecomposable projective semimodules if and only if~$S$ has a
finite set of orthogonal primitive idempotents whose sum is~$1_{S}$
(\cite[Thm.~3.3]{ikr:domlzs}).  In this case, every finitely generated
strongly projective $S$-semimodule is uniquely decomposed by these
orthogonal primitive idempotents (\cite[Cor.~3.4]{ikr:domlzs}).  The
following theorem allows us to express these results from the point of
view of $K$-theory.

\begin{thm}\label{thm:614}
  For any zerosumfree semiring~$S$, the following are equivalent:
  \begin{enumerate}[label=(\arabic*)]
  \item $S$ has a finite set of orthogonal primitive idempotents
    whose sum is~$1_{S}$;
  \item the monoid~$\mathcal{SV}(S)$ is cancellative and there exists a
    positive integer~$n$ such that $(SK_{0}(S), \widehat{S})
    \cong (\mathbb{Z}^{n}, (1, \dots, 1))$.
    \end{enumerate}
\end{thm}

\begin{proof}
  (1) $\!\Longrightarrow\!$ (2).  Let $\{e_{1}, \dots, e_{n}\}$ be a
  set of orthogonal primitive idempotents of~$S$ whose sum is~$1_{S}$.
  By \cite[Cor.~3.4]{ikr:domlzs}, every finitely generated strongly
  projective right $S$-semimodule~$P$ is uniquely written in the
  form $P \cong \bigoplus_{i=1}^{n}(e_{i} S)^{n_{i}}$ for some
  non-negative integers~$n_{i}$.  This implies that the monoid
  $\mathcal{SV}(S)$ is cancellative, and the map $f \colon SK_{0}(S)
  \to \mathbb{Z}^{n}$, defined by $f(\widehat{e_{i}S}) = \epsilon_{i}$ for
  all $i = 1, \dots, n$, is a group isomorphism, where $\{\epsilon_{1},
  \dots, \epsilon_{n}\}$ is the canonical basis of the free abelian group
  $\mathbb{Z}^{n}$.  Obviuosly, $S = \bigoplus_{i=1}^{n} e_{i}S$, hence,
  $f(\widehat{S}) = \sum_{i=1}^{n} \epsilon_{i} = (1, \dots, 1) \in
  \mathbb{Z}^{n}$.  From these observations, we get that
  $(SK_{0}(S), \widehat{S}) \cong (\mathbb{Z}^{n}, (1, \dots, 1))$.

  (2) $\!\Longrightarrow\!$ (1).  Let $f \colon (SK_{0}(S), \widehat{S})
  \to (\mathbb{Z}^{n}, (1, \dots, 1))$ be an isomorphism in~$\mathcal{P}$,
  and let $\{\epsilon_{1}, \dots, \epsilon_{n} \}$ be the canonical basis
  of the free abelian group $\mathbb{Z}^{n}$.  Then, for each $1 \le i
  \le n$, there is a unique element $\widehat{P_{i}} \in SK_{0}(S)^{+}$
  such that $f(\widehat{P_{i}}) = \epsilon_{i}$.  This implies that 
  \[ f(\widehat{\textstyle\bigoplus\limits_{i=1}^{n} P_{i}})
  = \smsum_{i=1}^{n} f(\widehat{P_{i}})
  = \smsum_{i=1}^{n} \epsilon_{i} = (1, \dots, 1)
  = f(\widehat{S}) \,, \]
  so that $\widehat{\bigoplus_{i=1}^{n} P_{i}} = \widehat{S}$, since~$f$ is
  injective.  By Lemma~\ref{lem:65}, there exists a positive integer~$m$
  such that $(\bigoplus_{i=1}^{n} P_{i}) \oplus S^{m} \cong S^{m+1}$, and hence,
  $\bigoplus_{i=1}^{n} P_{i} \cong S$, since the monoid $\mathcal{SV}(S)$
  is cancellative.  Consequently, there exists a set of orthogonal
  idempotents $\{ e_{1}, \dots, e_{n} \}$ of~$S$ such that
  $\sum_{i=1}^{n} e_{i} = 1$ and $P_{i} \cong e_{i} S$ for all~$i$.

  It is left to prove that each~$e_{i}$ is primitive.  To this end,
  assume that there exist idempotent elements $e, e' \in S$ such that 
  $e e' = 0 = e' e$ and $e_{i} = e + e'$.  We then have that 
  $e_{i} S = e S \oplus e' S$, and both $x := f(\widehat{eS})$ and
  $y := f(\widehat{e'S})$ are elements in $(\mathbb{Z}^{+})^{n}$,
  due to our hypothesis that $f(SK_{0}(S)^{+}) = (\mathbb{Z}^{+})^{n}$.
  From the equality \[ x + y = f(\widehat{e S}) + f(\widehat{e' S})
  = f(\widehat{e_{i} S}) = \epsilon_{i} \,, \]
  we immediately see that either $x = \epsilon_{i}$, $y = 0$ or
  $x = 0$, $y = \epsilon_{i}$.  Suppose that $y = 0$, so that
  $f(\widehat{e' S}) = 0$, which implies $\widehat{e' S}=0$ in
  $SK_{0}(S)$.  By Lemma~\ref{lem:65}, there exists a positive integer~$k$
  such that $e^{\prime} S \oplus S^k \cong S^{k}$, and therefore, $e' S=0$
  (since the monoid $\mathcal{SV}(S)$ is cancellative), that is,
  $e' = 0$.  A similar argument applies in the case $x = 0$.
  Thus~$e_{i}$ is a primitive idempotent, as desired.
\end{proof}

In order to apply~$SK_{0}$ to direct limits of semirings, we consider
direct limits in the category~$\mathcal{P}$.  Given a direct system of
objects $(G_{i}, u_{i})$ and morphisms~$f_{ij}$ in~$\mathcal{P}$, we
first form the direct limit~$G$ of the abelian groups~$G_{i}$.  For
each~$i$, let $g_{i} \colon G_{i} \to G$ denote the canonical
homomorphism.  Setting $G^{+} = \bigcup g_{i}(G_{i}^{+})$, we obtain a
cone~$G^{+}$ in~$G$, using which~$G$ becomes a pre-ordered abelian
group.  Notice that the homomorphisms~$g_{i}$ are all monotone.
Inasmuch as $f_{ij}(u_{i}) = u_{j}$ whenever $i \le j$, there is a
unique element $u \in G$ such that $g_{i}(u_{i}) = u$ for all~$i$, and
we observe that~$u$ is an order-unit in~$G$.  Thus, $(G, u)$ is an
object in~$\mathcal{P}$, and each~$g_{i}$ is a morphism from
$(G_{i}, u_{i})$ to $(G,u)$ in~$\mathcal{P}$.  It is easy to see that
$(G, u)$ is the direct limit of the $(G_{i},u_{i})$.

Finite products in the category~$\mathcal{P}$ can be formed in a
standard manner.  Namely, given objects $(G_1, u_1), \dots, (G_n, u_n)$
in~$\mathcal{P}$, we set $G = G_1 \times \ldots \times G_n$ and
$G^+ = G^+_1 \times \ldots \times G^+_n$.  Then~$G^+$ is a cone
in~$G$, using which~$G$ becomes a pre-ordered abelian group.  Also,
$u = (u_1, \dots , u_n)$ is an order-unit in~$G$, whence $(G, u)$ is
an object in~$\mathcal{P}$.  It is easy to check that $(G, u)$ is the
product of the $(G_i, u_i)$ in $\mathcal{P}$ (see, also,
\cite[p.~210]{g:vnrr} for details).

\begin{prop}[{cf.~\cite[Prop.~15.11, Prop.~15.13]{g:vnrr}}]\label{prop:615}
  The functor $(SK_{0}(-), \widehat{-}) \colon$ $\mathcal{SR} \to
  \mathcal{P}$ preserves direct limits and direct products.
\end{prop}

\begin{proof}
  Similarly as it was done in the proof of Proposition~5.4, we may show
  that the functor $(SK_{0}(-),\widehat{-})$ preserves direct limits;
  hence, we will not reproduce it here. (Also, using Proposition 4.10
  and repeating verbatim the proof of \cite[Prop.~15.11]{g:vnrr}, we
  immediately get the statement.)

  Using Proposition~\ref{prop:48} and repeating the proof of
  \cite[Prop.~15.13]{g:vnrr}, we obtain that the functor
  $(SK_{0}(-),\widehat{-})$ preserves direct products, and which, for
  the reader's convenience, we provide here. Namely, let~$S$ be the
  direct product of semirings $S_{1}, \dots, S_{n}$, and for each~$i$
  let $\phi_{i} \colon S \to S_{i}$ be the canonical projection.

  Let $(G, u)$ be the product of the $(SK_{0}(S_{i}), \widehat{S_{i}})$
  in $\mathcal{P}$, and for each~$i$ let~$p_{i}$ denote the canonical
  projection $(G, u) \to (SK_{0}(S_{i}), \widehat{S_{i}})$.  We have
  morphisms $SK_{0}(\phi_{i}) \colon (SK_{0}(S),\widehat{S}) \to
  (SK_{0}(S_{i}), \widehat{S_{i}})$ for each~$i$; hence, there is a unique
  morphism $f \colon (G, u) \to (SK_{0}(S), \widehat{S})$ such that
  $p_{i} f = SK_{0}(\phi_{i})$ for all~$i$.  We are going to prove
  that~$f$ is an isomorphism in $\mathcal{P}$.

  Given $x \in G^{+}$, there exists a finitely generated strongly
  projective right $S_{i}$-semimodule~$P_{i}$ for each~$i$ such that
  $p_{i}(x) = \widehat{P_{i}}$.  Then, by Proposition~\ref{prop:48} and
  Remark~\ref{rem:42}\,(3), $P = P_{1} \times \ldots \times P_{n}$ is a
  finitely generated strongly projective right $S$-semimodule such
  that $P \otimes_{S} S_{i} \cong P_{i}$ for all~$i$.  As a result, we
  get $\widehat{P} \in SK_{0}(S)^{+}$ such that
  \[ p_{i} f(\widehat{P}) = SK_{0}(\phi_{i}) (\widehat{P})
  = \widehat{P \otimes_{S} S_{i}} = \widehat{P_{i}} = p_{i}(x) \]
  for all~$i$, whence $f(\widehat{P}) = x$.  Thus, $f(SK_{0}(S)^{+})
  = G^{+}$.  Specially, it follows that~$f$ is surjective.

  Now consider any $\widehat{P} - \widehat{Q} \in \ker(f)$.  For all~$i$
  we have \[ \widehat{P \otimes_{S} S_{i}} - \widehat{Q \otimes_{S} S_{i}}
  = SK_{0}(\phi_{i})(\widehat{P} - \widehat{Q})
  = p_{i}f(\widehat{P} - \widehat{Q}) = 0; \]
  whence $(P \otimes_{S} S_{i}) \oplus S_{i}^{k_{i}} \cong (Q \otimes_{S}
  S_{i}) \oplus S_{i}^{k_{i}}$ for some a positive integer~$k_{i}$, by
  Lemma~\ref{lem:65}.  Then we obtain a positive integer $k = \max \{ k_{1},
  \dots, k_{n} \}$ such that 
  \[ (P \oplus S^{k}) \otimes_{S} S_{i} \cong (P \otimes_{S} S_{i})
  \oplus S_{i}^{k} \cong (Q \otimes_{S} S_{i}) \oplus S_{i}^{k} 
  \cong (Q \oplus S^{k}) \otimes_{S} S_{i} \]%
  for all~$i$.  Consequently, we get that $P \oplus S^{k} \cong Q \oplus
  S^{k}$, whence $\widehat{P} - \widehat{Q} = 0$, by Lemma~\ref{lem:65}.
  Thus $f$~is injective, and the proof is finished.
\end{proof}

We are going to present semiring analogs of Elliott's
theorem~\cite{elliott:otcoilososfa} for ultramatricial algebras over
an abitrary semifield.  In order to establish the main theorems of
this section, we need some preparatory facts and notions.

\begin{prop}\label{prop:616}
  Let~$S$ be a semisimple semiring with direct product representation
  \[ S \cong M_{n_{1}}(D_{1}) \times \cdots \times M_{n_{r}}(D_{r}) \,, \]
  where $D_{1}, \dots, D_{r}$ are division semirings.  For each 
  $1 \le j \leq r$ and $1 \le i \le n_{j}$, let $e_{ii}^{(j)}$ denote the
  $n_{i} \times n_{i}$ matrix unit in~$M_{n_{j}}(D_{j})$.  Then $SK_{0}(S)$
  is a free abelian group with basis $\{ \widehat{{e_{11}}^{(1)} S},
  \dots, \widehat{{e_{11}}^{(r)}S} \}$, and $(SK_{0}(S), \widehat{S})
  \cong (\mathbb{Z}^{r}, (n_{1}, \dots, n_{r}))$.
\end{prop}

\begin{proof}
  It follows from Propositions~\ref{prop:55} and~\ref{prop:615}, and
  Examples~\ref{exa:612}\,(2).
\end{proof}

Notice that by Corollary~\ref{cor:35} and Proposition~\ref{prop:616},
we immediately get that every free abelian group of finite rank may
appear as a~$SK_0(S)$ for some additively idempotent
congruence-semisimple semiring~$S$.  There are also pre-ordered
abelian groups not being free, which appear as a~$SK_0(S)$ for an
additively idempotent semiring~$S$.  We illustrate this by presenting
the following example.

\begin{exa}\label{exa:617}
  Fix an additively idempotent semifield~$F$ (\eg, the Boolean or the
  tropical semifield) and consider the semiring $S_{n} = M_{2^{n}}(F)$
  (for $n \ge 0$) of square matrices of order~$2^{n}$ over~$F$.  We may
  consider~$S_{n}$ as a subsemiring of~$S_{n+1}$ by identifying a
  $2^{n} \times 2^{n}$-matrix~$M$ with the $2^{n+1} \times 2^{n+1}$-matrix
  $\big( \begin{smallmatrix} M & 0 \\ 0 & M \end{smallmatrix} \big)$.
  In this way, we have a chain of additively idempotent semirings
  $S_{0} \subseteq S_{1} \subseteq S_{2} \subseteq \cdots$.
  Denoting by~$S$ the direct limit of the directed system
  $\{S_{n} \mid n \in \mathbb{Z}^{+} \}$, then~$S$ is an additively
  idempotent semiring.  Moreover, for all~$n$ we have
  \[ (SK_{0}(S_{n}), \widehat{S_{n}}) \cong (\tfrac{1}{2^{n}} \mathbb{Z}, 1) \]
  by Examples~\ref{exa:612}, so that using Proposition~\ref{prop:615}
  we see that $(SK_{0}(S), \widehat{S}) \cong (\varinjlim
  \frac{1}{2^{n}} \mathbb{Z}, 1)$; it is routine to check that
  $\varinjlim \frac{1}{2^{n}} \mathbb{Z} = \{ \frac{m}{2^{n}} \mid m
  \in \mathbb{Z} \text{ and } n \in \mathbb{Z}^{+} \}$.
\end{exa}

Example~\ref{exa:617} and the facts above motivate the following
natural question.

\begin{prob}
  Describe all pre-ordered abelian groups which can be
  either~$SK_{0}(S)$ or~$K_{0}(S)$ for some additively idempotent
  semiring~$S$.
\end{prob}

\begin{prop}\label{prop:618}
  Let~$F$ be a semifield and~$S$ a matricial $F$-algebra.  Then the
  canonical homomorphism $\jmath \colon SK_{0}(S) \to K_{0}(S)$ is
  injective.
\end{prop}

\begin{proof}
  It is easy to see that~$F$ is either a field or a zerosumfree semifield.
  If~$F$ is a field, then~$\jmath$ is obviously an isomorphism,
  since the functors~$SK_0$ and~$K_0$ are the same.  Assuming that~$F$
  is a zerosumfree semifield, write~$S$ in the form
  \[ S = M_{n_1}(F) \times \ldots \times M_{n_r}(F) \,, \]
  and for each $1 \le j \le r$ and $1 \le i \le n_j$ let $e_{ii}^{(j)}$
  be the $n_j \times n_j$ matrix unit in $M_{n_j}(F)$.  Then $SK_{0}(S)$ is
  a free abelian group with basis $\{ \widehat{{e_{11}}^{(1)} S}, \dots,
  \widehat{{e_{11}}^{(r)} S}\}$, by Proposition~\ref{prop:616}.  Therefore,
  the canonical homomorphism $\jmath \colon SK_{0}(S) \to K_{0}(S)$ is
  exactly given by $\jmath (\widehat{{e_{11}}^{(j)} S}) =
  [ e_{11}^{(j)} S]$ for all $j = 1, \dots, r$.

  Next we show that~$\jmath$ is injective for the case when~$F$ is the
  Boolean semifield~$\mathbb{B}$.  Indeed, assume that 
  $0 = \jmath (\sum_{j=1}^{r} k_{j} \widehat{{e_{11}}^{(j)} S})
  = \sum_{j=1}^{r}k_{j}[e_{11}^{(j)}S] \in K_{0}(S)$ for some nonnegative
  integers $k_{1}, \dots, k_{n}$.  We then have that
  $(\bigoplus_{j=1}^{r} (e_{11}^{(j)} S)^{k_{j}}) \oplus C \cong C$
  for some finitely generated projective right $S$-semimodule~$C$.
  Since~$S$ is finite and~$C$ is a finitely generated right
  $S$-semimodule, $C$ is also finite.  We deduce that
  $|\bigoplus_{j=1}^{r}(e_{11}^{(j)} S)^{k_{j}}| = 0$, whence $k_{j} = 0$
  for all~$j$.  Therefore, $\jmath$ is injective.

  We now consider the case when $F$ is an arbitrary zerosumfree
  semifield.  There exists a surjective semiring homomorphism
  $\pi \colon F \to \mathbb{B}$, defined by $\pi(0) = 0$ and
  $\pi(x) = 1$ for all $0 \ne x\in F$, which induces a surjective
  semiring homomorphism $\theta \colon S \to M_{n_1}(\mathbb{B})
  \times \ldots \times M_{n_r}(\mathbb{B}) =: T$.  We then have
  a commutative diagram:
  \[ \xymatrix{ SK_{0}(S) \ar[r]^{\jmath} \ar[d]_{SK_{0}(\theta)} &
    K_0(S) \ar[d]^{K_{0}(\theta)} \\
    SK_{0}(T) \ar[r]_{\jmath} & K_{0}(T) } \]
  Assume that $\jmath(\sum_{j=1}^{r} k_{j} \widehat{{e_{11}}^{(j)}S}) = 0
  \in K_{0}(S)$ for some integers $k_{1}, \dots, k_{n} \ge 0$.
  Since the above diagram is commutative and $\jmath \colon SK_{0}(T)
  \to K_{0}(T)$ is injective, we have that $SK_{0}(\theta) 
  (\sum_{j=1}^{r} k_{j} \widehat{{e_{11}}^{(j)} S}) = 0 \in SK_{0}(T)$;
  that means, $\sum_{j=1}^{r} k_{j} \widehat{{e_{11}}^{(j)}T} = 0 \in
  SK_{0}(T)$; hence $\bigoplus_{j=1}^{r} (e_{11}^{(j)} T)^{k_{j}}) \oplus
  T^{m}\cong T^{m}$ for some nonnegative integer~$m$.  This implies that
  $k_{j} = 0$ for all~$j$, whence $\jmath \colon SK_{0}(S)\to K_{0}(S)$
  is injective.
\end{proof}

Before stating the main results of this section, we name another
useful notion.

\begin{defn}
  A semiring~$S$ is said be have \emph{cancellation of projectives} if
  the monoid $\mathcal{SV}(S)$ is cancellative.
\end{defn}

\begin{rem}\label{rem:620}
  Let us note the following simple facts.
  \begin{enumerate}[label=(\arabic*)]
  \item Every semisimple semiring has cancellation of projectives.
    Indeed, this follows immediately from Proposition~\ref{prop:615}.
  \item It is not hard to see that a semiring~$S$ has cancellation of
    projectives if and only if the natural map $\mathcal{SV}(S) \to
    SK_{0}(S)$, $\overline{P}\longmapsto \widehat{P}$, is an injective
    monoid homomorphism (by using Lemma~\ref{lem:65}).  This shows that
    we may consider $\mathcal{SV}(S)$ to be the cone $SK_{0}(S)^{+}$ in
    $SK_{0}(S)$, for any semiring~$S$ having cancellation of projectives.
  \item Let~$F$ be a semifield and~$S$ a ultramatricial $F$-algebra. 
    Then~$S$ has cancellation of projectives.  Indeed, this follows
    immediately from Remark~\ref{rem:58}.
  \end{enumerate}
\end{rem}

Now we are ready to present the main theorems of this section.

\begin{thm}\label{thm:621}
  Let~$S$ and~$T$ be ultramatricial algebras over a semifield~$F$.
  Then $(SK_{0}(S), \widehat{S}) \cong (SK_{0}(T), \widehat{T})$
  if and only if $S \cong T$ as $F$-algebras.
\end{thm}

\begin{proof}
  ($\Longleftarrow $). It is clear.

  ($\Longrightarrow$).  By Remark~\ref{rem:620}~(2) and~(3), we obtain
  that $\mathcal{SV}(S) = SK_{0}(S)^{+}$ and $\mathcal{SV}(T) = SK_{0}(T)^{+}$.
  Let $f \colon (SK_{0}(S), \widehat{S}) \to (SK_{0}(T), \widehat{T})$ be
  an isomorphism in the category $\mathcal{P}$.  We then have that
  $f(SK_{0}(S)^{+}) = SK_{0}(T)^{+}$, and hence, $f(\mathcal{SV}(S))
  = \mathcal{SV}(T)$.  This implies that~$f$ induces an isomorphism
  from $(\mathcal{SV}(S), \overline{S})$ onto $(\mathcal{SV}(T),
  \overline{T})$ in category $\mathcal{C}$.  Now, applying
  Theorem~\ref{thm:510}, we immediately get that $S \cong T$ as
  $F$-algebras, as desired.
\end{proof}

Theorem~\ref{thm:621} is, in general, not valid for semisimple
semirings (even, additively idempotent ones).  For example,
$(SK_{0}(\mathbb{B}), \widehat{\mathbb{B}}) \cong (\mathbb{Z}, 1)
\cong (SK_{0}(\mathbb{T}), \widehat{\mathbb{T}})$, but the Boolean
semifield~$\mathbb{B}$ and the tropical semifield~$\mathbb{T}$ are not
isomorphic.  However, as a corollary of Theorem~\ref{thm:621}, the
following result permits us to classify zerosumfree
congruence-semisimple semirings in terms of their $SK_{0}$-groups.

\begin{thm}\label{thm:622}
  Let~$S$ and~$T$ be zerosumfree congruence-semisimple semirings.  
  Then $(SK_{0}(S), \widehat{S}) \cong (SK_{0}(T), \widehat{T})$
  if and only if $S \cong T$ as $\mathbb{B}$-algebras.
\end{thm}

\begin{proof}
  By Corollary~\ref{cor:35}, every zerosumfree semiring is a matricial
  $\mathbb{B}$-algebra.  Also, every matrical $\mathbb{B}$-algebra is
  an ultramatricial $\mathbb{B}$-algebra.  From these observations and
  Theorem~\ref{thm:621}, we immediately get the statement.
\end{proof}

Finally, according to Propositions~\ref{prop:63}, \ref{prop:615}
and~\ref{prop:618}, as well as Theorem~\ref{thm:621}, we obtain the
following result.

\begin{thm}[{cf.~\cite[Thm.~15.26]{g:vnrr}}]\label{thm:623}
  Let~$S$ and~$T$ be ultramatricial algebras over a semifield~$F$.
  Then, there exists a group isomorphism $f \colon K_{0}(S) \to
  K_{0}(T)$ such that $f([S]) = [T]$ and $f(K_{0}(S)^{+}) = K_{0}(T)^{+}$
  if and only if $S \cong T$ as $F$-algebras.
\end{thm}

\begin{proof}
  ($\Longleftarrow $). It is obvious.

  ($\Longrightarrow$). Assume that $f \colon K_{0}(S)\to K_{0}(T)$ is
  a group homomorphism such that $f([S]) = [T]$ and $f(K_{0}(S)^{+})
  = K_{0}(T)^{+}$.  Let~$S$ be the union of an ascending sequence
  $S_{1} \subseteq S_{2} \subseteq \cdots$ of matricial subalgebras,
  and~$T$ the union of an ascending sequence $T_{1} \subseteq T_{2}
  \subseteq \cdots$ of matricial subalgebras.  We then have a
  commutative diagram in the category of abelian groups as follows:
  \[ \xymatrix{ SK_{0}(S_{1}) \ar[r] \ar[d]_{\jmath} & SK_{0}(S_{2})
    \ar[r] \ar[d]_{\jmath} & SK_{0}(S_{3}) \ar[r] \ar[d]_{\jmath} & \cdots \\
    K_{0}(S_{1}) \ar[r] & K_{0}(S_{2}) \ar[r] & K_{0}(S_{3}) \ar[r] & \cdots } \]
  
  From Proposition~\ref{prop:618} we have that the homomorphisms
  $\jmath \colon SK_{0}(S_{n}) \to K_{0}(S_{n})$ are injective for
  all~$n$.  Then, taking into account Propositions~\ref{prop:63}
  and~\ref{prop:615}, we may consider~$SK_{0}(S)$ to be a subgroup
  of~$K_{0}(S)$; and similarly, $SK_{0}(T)$ may be considered as a
  subgroup of~$K_{0}(T)$.  Then, by our hypothesis, we immediately get
  that~$f$ induces an isomorphism from $(SK_{0}(S), \widehat{S})$ onto
  $(SK_{0}(T), \widehat{T})$ in the category~$\mathcal{P}$, whence we
  obtain the statement by Theorem~\ref{thm:621}.
\end{proof}

Motivated by Elliott's program and Theorems~\ref{thm:622}
and~\ref{thm:623}, we finish this article by posting the following
problem.

\begin{prob}
  Describe all additively idempotent semirings for which the
  $SK_{0}$-groups and $K_{0}$-groups are complete invariants.
\end{prob}

\end{document}